\documentclass[12pt]{amsart}

\usepackage{eucal}
\usepackage{mathrsfs}
\usepackage{MnSymbol}
\usepackage{tikz}
\usepackage{paralist}
\usepackage{caption}
\usepackage{subcaption}
\usepackage{graphicx}
\usepackage{color}

\newcommand{\rev}[1]{{#1}}
\DeclareMathOperator{\aarg}{arg}
\newcommand{\uparr}{\tikz\draw[->,line width=0.05cm] (0,0)--(0,0.3);}
\newcommand{\rtarr}{\tikz\draw[->,line width=0.05cm] (0,0)--(0.3,0);}
\newcommand{\bullt}{\tikz \fill (0,0) circle (2.5pt);}
\newcommand{\eint}{e}

\newcommand{\QQQ}{\mathcal{Q}}

\newcommand{\ZZZ}{\mathbb{Z}}

\newcommand{\veps}{\varepsilon}

\newcommand{\vw}{\vec{w}}
\newcommand{\vn}{\vec{n}}
\newcommand{\vl}{\vec{l}\,}
\newcommand{\vj}{\vec{\jmath}}

\newcommand{\vx}{\vec{x}}
\newcommand{\pjt}{\psi_{t,\vj}}
\newcommand{\pjsl}{\psi_{s,\vj+\vl}}

\newcommand{\Dtsl}{D_{t,s,\vl}}
\newcommand{\vk}{\vec{k}}
\newcommand{\vv}{\vec{v}}

\newcommand{\grad}{\vec\nabla}
\def\d{\partial}

\newtheorem{theorem}{Theorem}[section]

\theoremstyle{definition}

\theoremstyle{remark}
\newtheorem{remark}[theorem]{Remark}

\def\d{\partial}
\def\em{\it}

\newcommand{\mysl}[1]{{\textit{\textsf{#1}}}}

\newcommand{\hu}{\hat{u}}
\newcommand{\hphi}{\hat{\phi}}
\newcommand{\ha}{\hat{a}}
\newcommand{\he}{\hat{e}}

\newcommand{\UU}{\mysl{u}}
\newcommand{\UUh}{\UU_{h}}
\newcommand{\VVh}{\VV_{\!h}}
\newcommand{\VV}{\mysl{v}\,}
\newcommand{\WW}{\mysl{w}\,}
\newcommand{\EE}{\mysl{e}\,}
\newcommand{\FF}{\mysl{f}\,}
\newcommand{\QQ}{\mysl{q}\,}
\newcommand{\RR}{\mysl{r}\,}
\newcommand{\ZZ}{\mysl{z}}
\newcommand{\XX}{\mysl{x}}

\newcommand{\hqq}{\mathsf{\hat{\QQ}}}
\newcommand{\hrr}{\mathsf{\hat{\RR}}}
\newcommand{\hw}{\hat w}
\newcommand{\hpsi}{\hat \psi}

\newcommand{\cj}[1]{\overline{{#1}}}

\newcommand{\re}{\mathrm{Re}}
\newcommand{\im}{\mathrm{Im}}

\newcommand{\ii}{\hat{\imath}}

\newcommand{\vuls}{\vu_h^{\mathrm{ls}}}
\newcommand{\phils}{\phi_h^{\mathrm{ls}}}
\newcommand{\vuh}{{\vec{u}_h}}

\newcommand{\tr}{\mathrm{tr}_h}

\newcommand{\Hdiv}[1]{H(\mathrm{div},#1)}

\newcommand{\ip}[1]{ \langle{#1}\rangle}
\newcommand{\iip}[1]{\llangle {#1} \rrangle}   

\newcommand{\og}{\omega}

\newcommand{\om}{\varOmega}
\newcommand{\oh}{{\varOmega_h}}

\newcommand{\vq}{{\vec{q}\,}}
\newcommand{\vr}{{\vec{r}\,}}
\newcommand{\vu}{{\vec{u}\,}}

\newcommand{\ww}{{\vec{w}\,}}
\newcommand{\vb}{{\vec{b}}}

\newcommand{\vtau}{{\vec{\tau}}}

\newcommand{\dive}{\vec{\nabla}\cdot}

\newcommand{\CCC}{\mathbb{C}}

\title[Dispersion analysis for DPG methods] {Dispersive and
  dissipative errors in the DPG method with scaled norms for 
  Helmholtz equation}

\author[Gopalakrishnan]{J. Gopalakrishnan}
\address{Portland State University, PO Box 751, Portland, OR 97207-0751, USA.}
\email{gjay@pdx.edu}

\author[Muga]{I. Muga}
\address{Pontificia Universidad Cat{\'o}lica de Valpara{\'i}so, Casilla 4059, Valpara{\'i}so, Chile.}
\email{ignacio.muga@ucv.cl}

\author[Olivares]{N. Olivares}
\address{Portland State University, PO Box 751, Portland, OR 97207-0751, USA.}
\email{nmo@pdx.edu}

\dedicatory{This paper is dedicated to Leszek Demkowicz on the occasion of his 60${\,}^\mathrm{th}$ birthday.}

\thanks{This work was partially supported by the NSF under grant
  DMS-1211635, by the AFOSR under grant FA9550-12-1-0484, and by the
  FONDECYT project 1110272.}

\keywords{least-squares, dispersion, dissipation,
  quasioptimality, resonance, stencil}

\subjclass[2010]{65N30,35J05}

\begin{document}

\maketitle

\begin{abstract}
  We consider the discontinuous Petrov-Galerkin (DPG) method, where
  the test space is normed by a modified graph norm. The modification
  scales one of the terms in the graph norm by an arbitrary positive
  scaling parameter. Studying the application of the method to the
  Helmholtz equation, we find that better results are obtained, under
  some circumstances, as the scaling parameter approaches a limiting
  value. We perform a dispersion analysis on the multiple interacting
  stencils that form the DPG method. The analysis shows that the
  discrete wavenumbers of the method are complex, explaining the
  numerically observed artificial dissipation in the computed wave
  approximations. Since the DPG method is a nonstandard least-squares
  Galerkin method, we compare its performance with a standard
  least-squares method.
\end{abstract}

\section{Introduction}   \label{sec:introduction}

Discontinuous Petrov-Galerkin (DPG) methods were introduced
in~\cite{DemkoGopal10a,DemkoGopal11}.  The DPG methods minimize a
residual norm, so they belong to the class of least-squares Galerkin
methods~\cite{BocheGunzb09a,CaiLazarMante94,FixGunzb80}, although the
functional setting in DPG methods is nonstandard. In this paper, we
introduce an arbitrary parameter $\veps>0$ into the definition of the
norm in which the residual is minimized. We study the properties of
the resulting family of DPG methods when applied to the Helmholtz
equation.

The DPG framework has already been applied to the Helmholtz equation
in~\cite{DemkoGopalMuga11a}. An error analysis with optimal error
estimates was presented there. There are two major differences in the
content of this paper and~\cite{DemkoGopalMuga11a}. The first is the
introduction of the above mentioned parameter, $\veps$. When
$\veps=1$, the method here reduces to that
in~\cite{DemkoGopalMuga11a}. The use of such scaling parameters was
already advocated in~\cite{DemkoGopalNiemi12} based on numerical
experience. In this paper, we shall provide a theoretical basis for
its use.  The second major difference with~\cite{DemkoGopalMuga11a} is
that in this contribution we perform a dispersion analysis of the DPG
method with the~$\veps$ scaling. We thus discover several important
properties of the method as $\veps$ is varied.

Least-squares Galerkin methods are popular methods in scientific
computation~\cite{BocheGunzb09a,Jiang98}.  They yield Hermitian
positive definite systems, notwithstanding the indefiniteness of the
underlying problem. Hence they are attractive from the point of view
of solver design and many works have focused on this
subject~\cite{Lee99,LeeManteMcCor00}. However, as we shall shortly see
in detail, for wave propagation problems, they yield solutions with
heavy artificial dissipation.  Since the DPG method is of the
least-squares type, it also suffers from this problem. One of the
goals of this paper is to show that by means of the $\veps$-scaling,
we can rectify this problem to some extent.

To explain this issue further, let us fix the specific boundary value
problem we shall consider. Let $A: \Hdiv \om \times H^1(\om) \to 
L^2(\om)^N \times L^2(\om)$ denote the Helmholtz wave operator defined
by
\begin{equation}
  \label{eq:A}
  A ( \vv, \eta) = ( \ii \og \vv + \grad \eta, \ii \og \eta + \dive \vv).
\end{equation}
Here $\ii$ denotes the imaginary unit, $\og$ is the wavenumber, and
$\om$ is a bounded open connected domain with Lipschitz boundary.  All
function spaces in this paper are over the complex field $\CCC$.  The
Helmholtz equation takes the form $ A (\vu,\phi) = \FF, $ for some
$\FF\in L^2(\om)^N \times L^2(\om)$. Although, we consider a general
$\FF$ in this paper, in typical applications, $\FF = (\vec 0, f)$ with
$f\in L^2(\om)$, in which case, eliminating the vector component
$\vu$, we recover the usual second order form of the Helmholtz
equation,
\[
-\Delta \phi -\og^2 \phi  = \ii \og f, \qquad \text{ on } \om.
\]
This must be supplemented with boundary conditions.  The DPG method
for the case of the impedance boundary conditions $ \ii\og\phi + \d
\phi/\d n= 0 $ on $\d\om$ was discussed in~\cite{DemkoGopalMuga11a},
but other boundary conditions are equally well admissible.  In the
present work, we consider the Dirichlet boundary condition
\begin{equation}  
  \label{eq:2ndorder-bc}
  \phi  = 0,  \quad \text{ on } \d\om.
\end{equation}
To deal with this boundary condition, we will need the space
\begin{equation}    
  \label{eq:R}
  R = \Hdiv\om \times H_0^1(\om),
\end{equation} 
Thus, our boundary value problem reads as follows: 
\begin{equation}
  \label{eq:bvp-intro}
  \text{Find } (\vu,\phi) \in R \text{ satisfying } A(\vu,\phi) = \FF.
\end{equation}
It is well known~\cite{Ihlen98} that except for $\omega$ in
$\Sigma$, an isolated countable set of real values, this problem has a unique
solution. We assume henceforth that $\omega$ is not in $\Sigma$.

Before studying the DPG method for~\eqref{eq:bvp-intro}, it is instructive
to examine the simpler $L^2$ least-squares Galerkin method.  Set $R_h
\subset R$ to the Cartesian product of the lowest order Raviart-Thomas
and Lagrange spaces, together with the boundary condition in $R$.  The
method finds $(\vuls,\phils) \in R_h$ such that
\begin{equation}
  \label{eq:ls}
  (\vuls,\phils)=\aarg\min_{\WW\in R_h} \|  \FF -  A \WW \|.
\end{equation}
Throughout, $\| \cdot \|$ denotes the $L^2(\om)$ norm, or the natural
norm in the Cartesian product of several $L^2(\om)$ component spaces.
The method~\eqref{eq:ls} belongs to the so-called
FOSLS~\cite{CaiLazarMante94} class of methods.

Although~\eqref{eq:ls} appears at first sight to be a reasonable
method, computations yield solutions with artificial
dissipation. For example, suppose we use~\eqref{eq:ls}, appropriately
modified to include nonhomogeneous boundary conditions, to approximate
a plane wave propagating at angle $\theta=\pi/8$ in the unit square. A
comparison between the real parts of the exact solution (in
Figure~\ref{fig:dissip-exact}) and the computed solution (in
Figure~\ref{fig:dissip-ls}) shows that the computed solution
dissipates at interior mesh points. The same behavior
is observed for the lowest order DPG method with $\veps=1$ in
Figure~\ref{fig:dissip-dpg1} (see \S\ref{ssec:inexact} for the
definition of $r$ therein and Section~\ref{sec:stencil} for a full
discussion of the lowest order DPG method). The same method with
$\veps=10^{-6}$ however gave a solution (in
Figure~\ref{fig:dissip-dpg1e-6}) that is visually indistinguishable
from the exact solution.  Note that, for the DPG method with
$\veps=1$, the numerical results presented in~\cite{DemkoGopalMuga11a}
show much better performance, because slightly higher order spaces were
used there. Instead, in this paper, we have chosen to study the DPG
method with the lowest possible order of approximation spaces to
reveal the essential difficulties with minimal computational effort.

\begin{figure}   
  \centering     
  \begin{subfigure}{0.45\textwidth}
    \begin{center}
      \includegraphics[width=\textwidth]{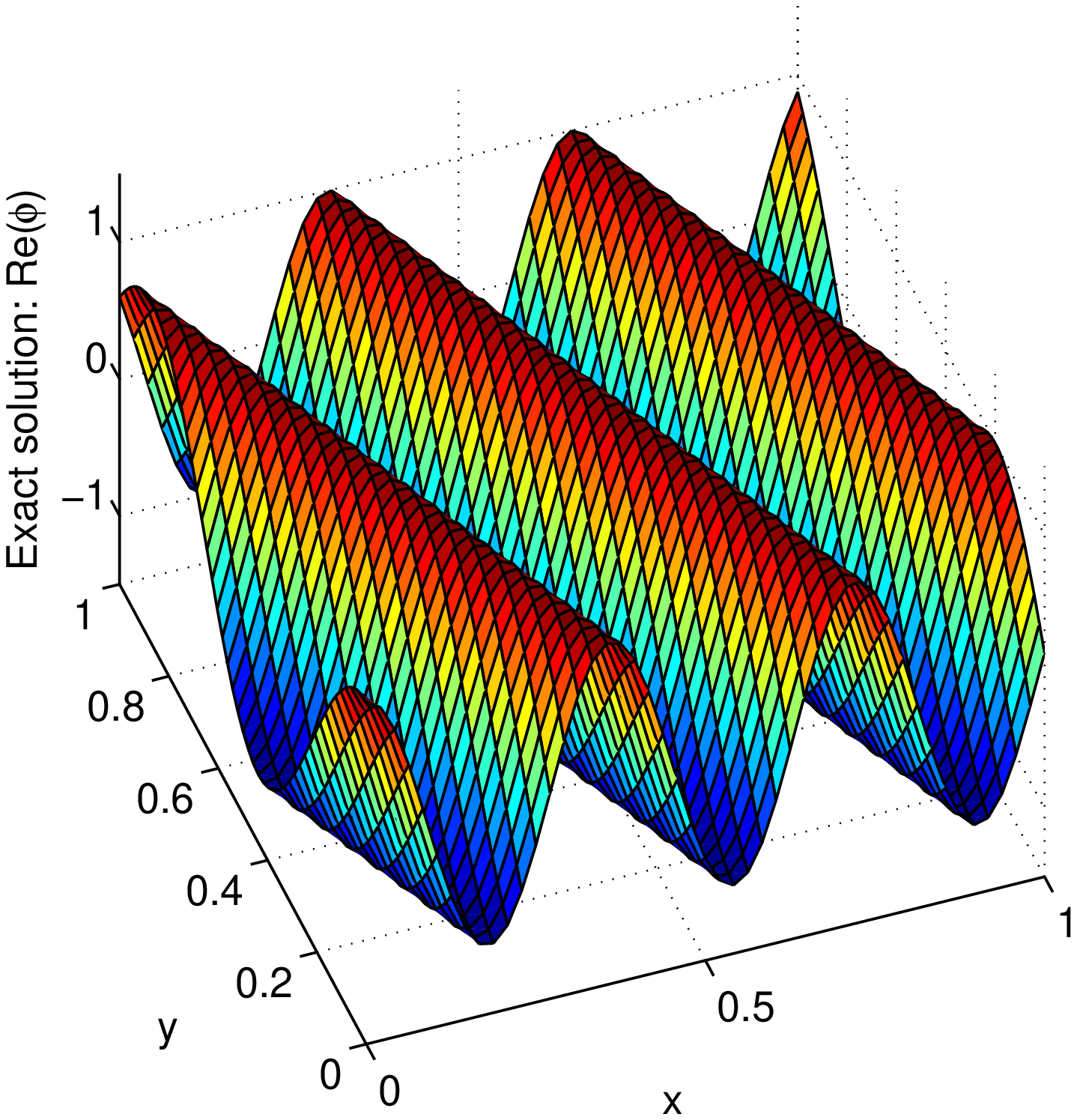}
      \caption{A plane wave propagating at angle $\pi/8$}
      \label{fig:dissip-exact}
    \end{center}
  \end{subfigure}\quad
  \begin{subfigure}{0.45\textwidth}
    \begin{center}
      \includegraphics[width=\textwidth]{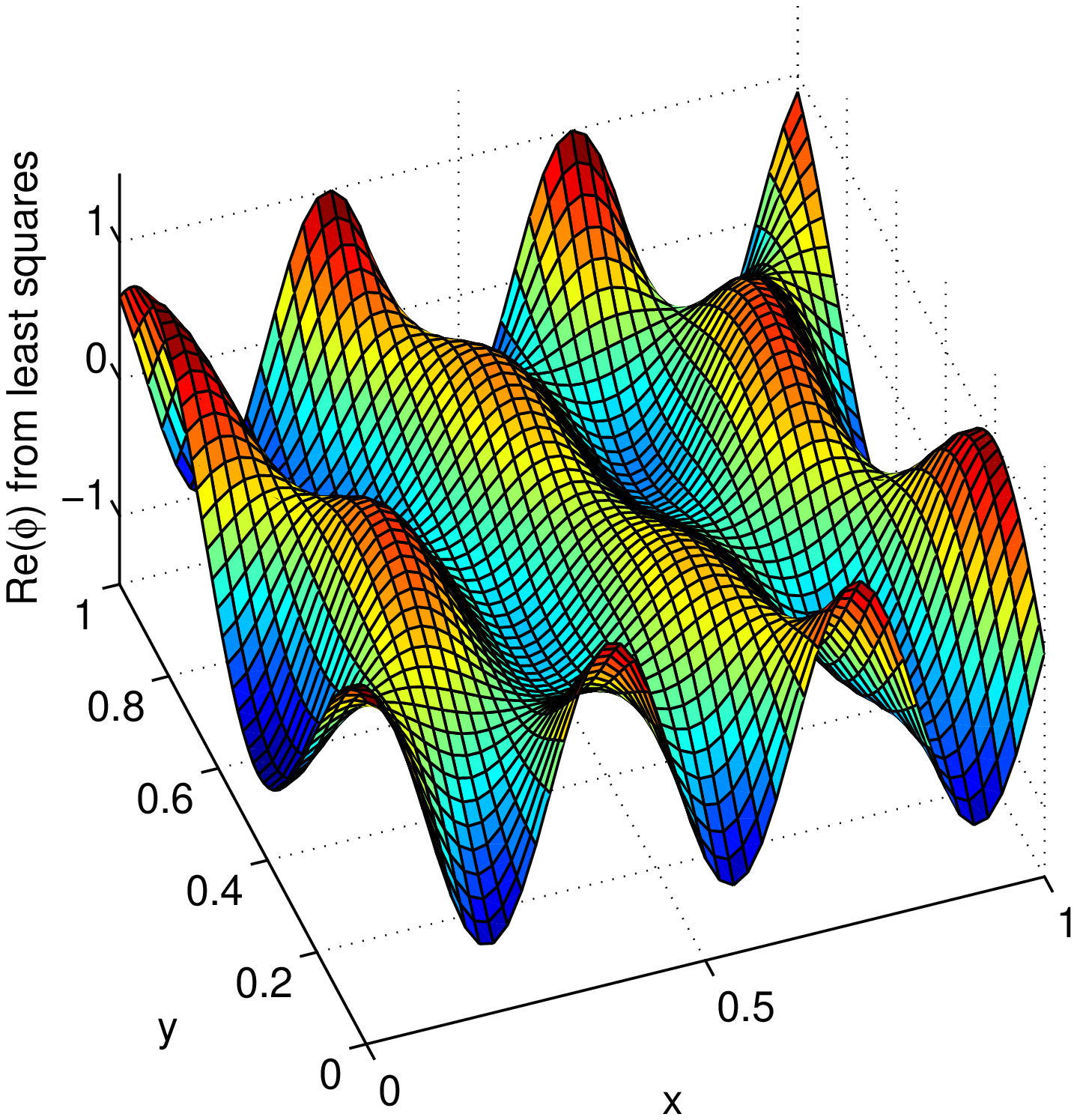}
      \caption{$L^2$ least-squares solution}
      \label{fig:dissip-ls}
    \end{center}
  \end{subfigure}
  \begin{subfigure}{0.45\textwidth}
    \begin{center}
      \includegraphics[width=\textwidth]{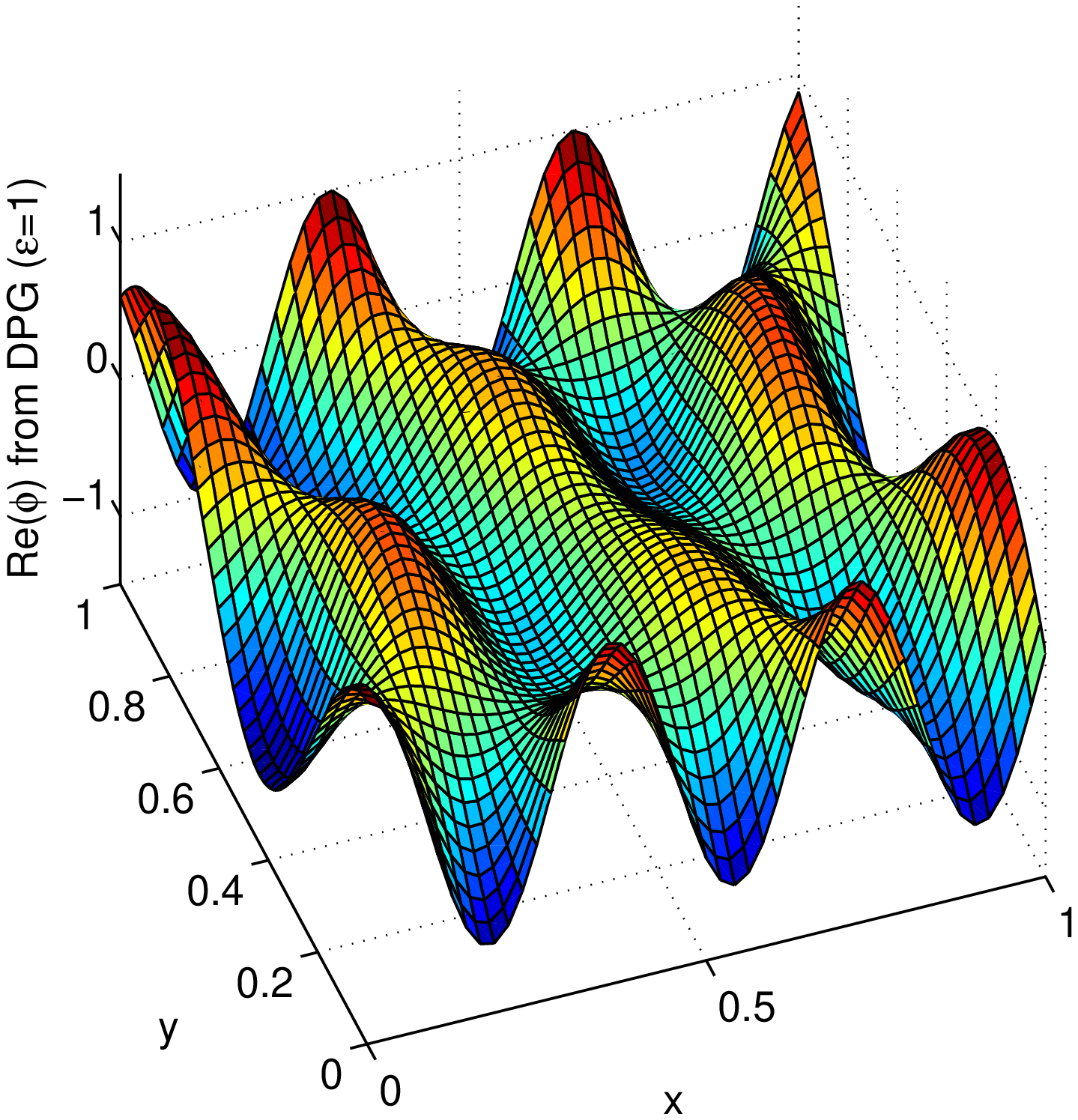}
      \caption{Numerical traces from the lowest order DPG method with
        $\veps=1$ and $r=3$}
      \label{fig:dissip-dpg1}
    \end{center}
  \end{subfigure}\quad
  \begin{subfigure}{0.45\textwidth}
    \begin{center}
      \includegraphics[width=\textwidth]{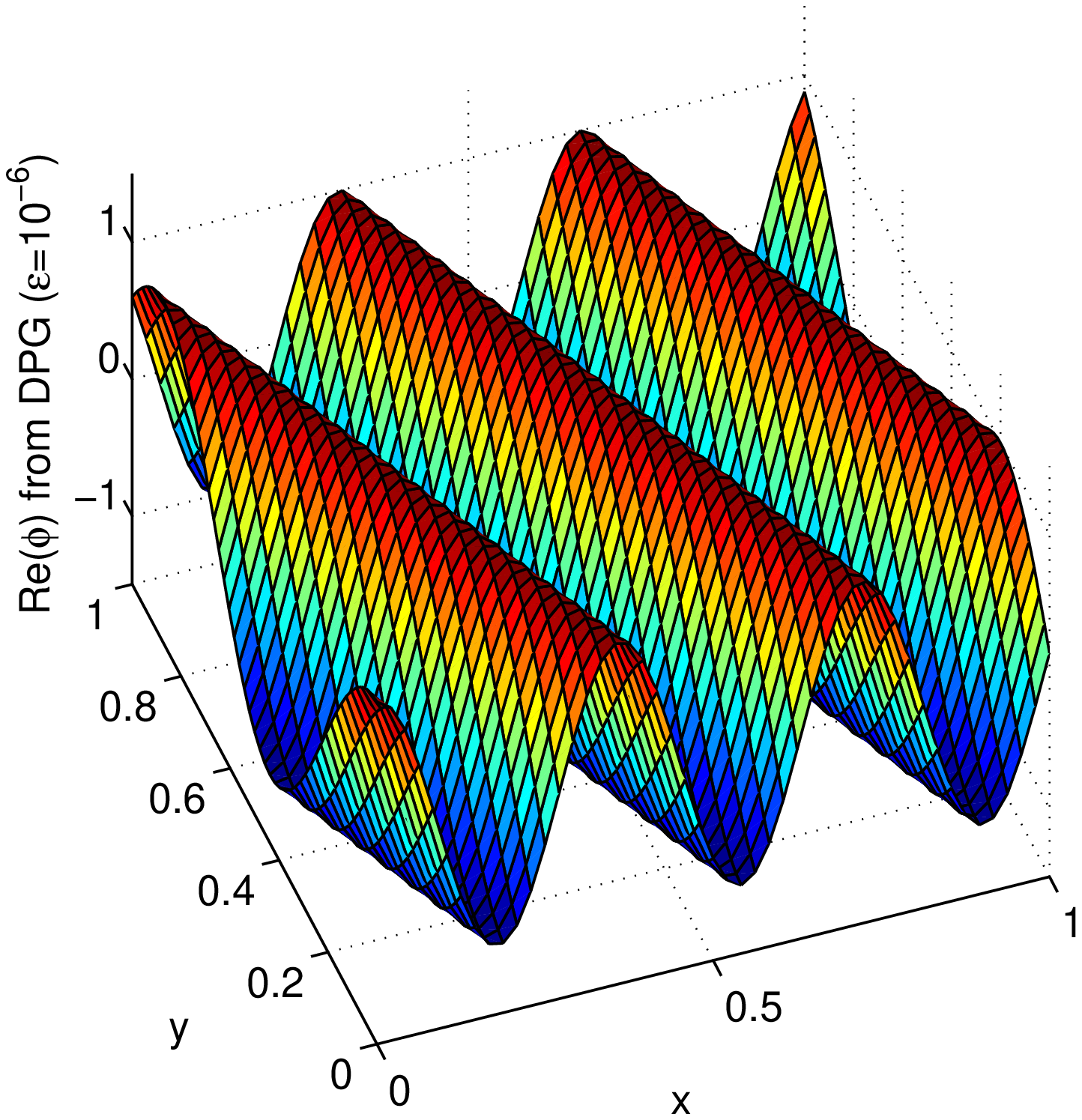}
      \caption{Numerical traces from the lowest order DPG method with
        $\veps=10^{-6}$ and $r=3$}
      \label{fig:dissip-dpg1e-6}
    \end{center}
  \end{subfigure}
  \caption{Approximations to a plane wave computed using a uniform
    mesh of square elements of size $h=1/48$ (about sixteen elements per
    wavelength). Artificial dissipation is visible
    in Figures~\ref{fig:dissip-ls} and~\ref{fig:dissip-dpg1}.}
  \label{fig:dissip}
\end{figure}

The situation in Figures~\ref{fig:dissip-ls} and~\ref{fig:dissip-dpg1}
improves when more elements per wavelength are used. This is not
surprising in view of the asymptotic error estimates of the
methods. To give an example of such an error estimate, consider the
case of the impedance boundary conditions considered
in~\cite{DemkoGopalMuga11a}. It is proven there that there is a
constant $C>0$, independent of $\og$ and mesh size $h$, such that the
lowest order DPG solution $(\vu_h,\phi_h)$ satisfies
\begin{equation}
  \label{eq:16}
\| \vu -\vu_h \| + \| \phi - \phi_h \| \le C \og^2 h  
\end{equation}
for a plane wave solution.  A critical ingredient in this analysis is
the estimate
\begin{equation}
  \label{eq:15}
  \| \WW \| \le C' \| A \WW \|, 
\end{equation}
which, as shown in~\cite[Lemmas~4.2 and 4.3]{DemkoGopalMuga11a}, holds
for all $\WW$ in the analogue of $R$ with impedance boundary
conditions. Although the analysis in~\cite{DemkoGopalMuga11a} was for
the impedance boundary condition, similar techniques apply to the
Dirichlet boundary condition as well, leading to~\eqref{eq:16}.  As more elements per wavelength are used,
$\og h$ decreases, so~\eqref{eq:16} guarantees that the situation in
Figure~\ref{fig:dissip-dpg1} will improve.

The analysis for the $L^2$ least-squares method is easier than the
above-mentioned DPG analysis. Indeed, by~\eqref{eq:ls}, $ \| \FF -
A(\vuls,\phils) \| \le \| A ( \vu - \vw_h,\phi-\psi_h)\|$ for any
$(\vw_h, \psi_h) \in R_h$. Hence, applying~\eqref{eq:15} to the error
$\EE = (\vu - \vuls, \phi - \phils)$ and noting that the residual is
$A \EE = \FF - A(\vuls,\phils)$, we obtain $\| \EE \| \le C' \| A (
\vu - \vw_h,\phi-\psi_h)\|$. By standard approximation estimates, we 
then conclude that there is a $C>0$ independent of $\og$ and $h$ such
that
\begin{equation}
  \label{eq:17}
  \| \vu -\vuls \| + \| \phi - \phils \| \le C \og^2 h.  
\end{equation}
This simple technique of analysis of $L^2$-based least-squares methods
is well-known (see~e.g., \cite[pp.~70--71]{Jiang98}). As
with~\eqref{eq:16}, the estimate~\eqref{eq:17} implies that as the
number of elements per wavelength is increased, $\og h$ decreases, and
the situation in Figure~\ref{fig:dissip-ls} must improve.

Yet, Figures~\ref{fig:dissip-ls} and~\ref{fig:dissip-dpg1} show that
these methods fail to be competitive with standard methods in accuracy
for small number of elements per wavelength. The figures also
illustrate one of the difficulties with asymptotic error estimates
like~\eqref{eq:15} and~\eqref{eq:17}. Having little knowledge of the
size of $C$, we cannot predict the performance of the method on coarse
meshes. Motivated by this difficulty, one of the theorems we present
(Theorem~\ref{thm:eps}) will give a better idea of the constant
involved as $\veps \to 0$. Also note that the above indicated error
analyses does not give us a quantitative measure of differences in
wave speeds between the computed and exact waves.  This motivates the
dispersion analysis we present in this paper, which will address the
issue of wave speed discrepancies.

We should note that there are alternative methods of the least-squares
type that exhibit better performance than the standard $L^2$-based
least squares method. Some are based on adding further terms to the
residual to be minimized (e.g., to control the curl of the vector
equation~\cite{Lee99}).  Another avenue explored by others, and closer
to the subject of this paper, is the idea of minimizing the residual
in a dual norm~\cite{BrambKolevPasci05,BrambLazarPasci97}.  The main
difference with our method is that our dual norms are locally
computable in contrast to their nonlocal norms. This is achieved by
using an ultraweak variational setting. The domain and codomain of the
operator in the least-squares minimization associated to the DPG
method are nonstandard, as we shall see next.

\section{The DPG method for the Helmholtz equation}
\label{sec:dpg}

In this section, we briefly review the method for the Helmholtz
equation introduced in~\cite{DemkoGopalMuga11a}. We then show exactly
where the parameter $\veps$ is introduced to get the variant of the
method that we intend to study.

Let $\oh$ be a disjoint partitioning of $\om$ into open elements $K$
such that $\overline{\om}=\cup_{K\in\oh}\overline{K}$. The shape of
the mesh elements in $\oh$ is unimportant for now, except that we
require their boundaries $\d K$ to be Lipschitz so that traces make
sense. Let
\begin{equation}
  \label{eq:V}
  V = \Hdiv \oh \times H^1(\oh),
\end{equation}
where
\begin{align*}
 \Hdiv\oh
 & = \{ \vtau: \; \vtau|_K \in \Hdiv {K},
 \;\forall K \in \oh\},
 \\ \nonumber 
 H^1(\oh)
 & = \{ v: \; v|_K \in H^1(K), \;\forall K \in \oh\}.
\end{align*}
Let $A_h: V\to L^2(\om)^N \times L^2(\om)$ be defined in the same
way as $A$ in~\eqref{eq:A}, except the derivatives are taken element
by element, i.e., on each $K\in \oh$, we have $A_h ( \vv, \eta)|_K = (
\ii \og \vv|_K + \grad \eta|_K, \ii \og \eta|_K + \dive \vv|_K).$

\subsection{Integration by parts}

The following basic formula that we shall use is obtained simply by
integrating by parts each of the derivatives involved:
\begin{equation}
  \label{eq:1}
  \int_D A (\ww, \psi) \cdot \cj{(\vv,\eta)}
  = 
  -\int_D (\ww, \psi) \cdot \cj{A  (\vv,\eta)}
  + \int_{\d D} (\ww \cdot \vn) \,\cj{\eta} 
  + \int_{\d D} \psi \,\cj{{(\vv\cdot\vn)}}, 
\end{equation}
for smooth functions $(\ww, \psi)$ and $ (\vv,\eta)$ and domains $D$
with Lipschitz boundary. Above, overlines denote complex conjugations
and the integrals use the appropriate Lebesgue measure.  Note that we
use the notation $\vn$ throughout to generically denote the outward
unit normal on various domains -- the specific domain will be clear
from context -- e.g., in~\eqref{eq:1}, it is $D$.  Introducing the
following abbreviated notations for tuples $\WW = (\ww,\psi)$ and
$\VV=(\vv, \eta)$,
\begin{align*}
\ip{\WW, \VV}_h 
& = \sum_{K\in\oh}  \int_K \ww\cdot \cj{\vv} + \psi \,\cj{\eta},
\\
\iip{ \WW, \VV }_h
& =  
\sum_{K\in\oh}  
\int_{\d K} {(\ww \cdot \vn)} \,\cj{\eta} 
+ 
\int_{\d K} \psi \,\cj{{(\vv\cdot\vn)}}\,,  
\end{align*}
we can rewrite~\eqref{eq:1}, applied element by element, as
\begin{equation}
  \label{eq:2}
  \ip{ A \WW, \VV }_h = -\ip{\WW, A_h \VV}_h + \iip{ \WW, \VV}_h.
\end{equation}
By density, \eqref{eq:2} holds for all $\WW \in \Hdiv\om \times
H^1(\om)$ and all $\VV \in V$.  \rev{Then, $\iip{\cdot,\cdot}_h$ must
  be interpreted using the appropriate duality pairing as the last term
  in~\eqref{eq:2} contains interelement traces on $\d\oh = \{
  \d K: K \in \om_h\}$.}

It will be convenient to introduce notation
for such traces: Define
\[
\tr : \Hdiv\om \times H^1(\Omega) \to 
\prod_K H^{-1/2}(\d K)\vn \times H^{1/2}(\d K)
\]
as follows. For any $( \ww,\psi) \in \Hdiv\om \times H^1(\Omega)$, the
restriction of $\tr( \ww,\psi)$ on the boundary of any mesh element
$\d K$ takes the form $((\ww\cdot \vn)\vn|_{\d K}, \psi|_{\d K}) \in
H^{-1/2}(\d K)\vn \times H^{1/2}(\d K).$ \rev{Although the
  meaning of $H^{-1/2}(\d K)\vn$ is more or less self-evident, to
  include a proper definition, first let $Z$ denote the space of all
  functions of the form $\xi \vn$ where $\xi$ is in $H^{1/2}(\d K)$,
  normed by $\| \xi \vn \|_Z = \| \xi \|_{H^{1/2}(\d K)}.$ Let $Z'$
  denote the dual space of $Z$.  Now, consider the map $M \vq =
  (\vq\cdot\vn) \vn|_{\d K}$, defined for smooth functions $\vq$ on
  $\bar K$.  Since 
  \[
  \int_{\d K} M \vq \cdot \xi \vn = \int_{\d K}(\vq\cdot\vn)
  \xi
  \] 
  (the left and right hand sides extend to duality pairings in $Z$ and
  $H^{1/2}(\d K)$, respectively), the standard trace theory implies
  that $M$ can be extended to a continuous linear operator $M: \Hdiv K
  \to Z'$. The range of $M$ is what we denoted by ``$H^{-1/2}(\d
  K)\vn$.'' Throughout this paper, functions in $H^{-1/2}(\d K)\vn$
  appear together with a dot product with $\vn$, so we could equally
  well consider the standard space $H^{-1/2}(\d K)$, but the notation
  simplifies with the former. In particular, with this notation, $\tr(
  \ww,\psi)$ is a single-valued function on the element interfaces
  since $( \ww,\psi)$ is globally in $\Hdiv\om \times H^1(\Omega)$.}

\subsection{An ultraweak formulation}

The boundary value problem we wish to approximate is~\eqref{eq:bvp}.
Recall the definition of $R$ in~\eqref{eq:R}.  To deal with the
Dirichlet boundary condition, we will need the trace space
\begin{equation}
  \label{eq:Q}
  Q = \tr(R).
\end{equation}
To derive the DPG method for
\begin{subequations}
  \label{eq:bvp}
  \begin{align}
    \label{eq:bvp-1}
    A (\vu,\phi) &  = \FF, && \text{ on } \om, \\
    \phi & =0, && \text{ on } \d \om,
  \end{align}
\end{subequations}
we use the integration parts by formula~\eqref{eq:2} to get 
\[
-\ip{ (\vu,\phi), A_h(\vv,\eta) }_h + \iip{ \tr(\vu,\phi), (\vv,\eta)}_h
= \ip{ \FF, (\vv,\eta) }_h
\]
for all $(\vv,\eta) \in V$. Now we let the trace $\tr(\vu,\phi)$ be an
independent unknown $(\hu,\hphi)$ in $Q$.  Defining the bilinear form
$b( (\vu,\phi,\hu,\hphi),(\vv,\eta)) = -\ip{ (\vu,\phi), A_h(\vv,\eta)
}_h + \iip{ (\hu,\hphi), (\vv,\eta)}_h$, we obtain the ultraweak
formulation of~\cite{DemkoGopalMuga11a}: Find
$\UU=(\vu,\phi,\hu,\hphi)$ in
\[
U= L^2(\om)^N \times L^2(\om) \times Q
\]
satisfying
\begin{equation}
  \label{eq:weakform}
  b(\UU,\VV)  =  \ip{ \FF, \VV}_h,
  \qquad
  \forall \VV \in V.
\end{equation}
The wellposedness of this formulation was proved
in~\cite{DemkoGopalMuga11a} for the case of impedance boundary
conditions. As is customary, we refer to the solution component $\hu$
as the {\em numerical flux} and $\hphi$ as the {\em numerical trace}.

\subsection{The $\veps$-DPG method} 

Let $U_h \subset U$ be a finite dimensional trial space.  The
{\em DPG method} finds $\UUh$ in $U_h$ satisfying
\begin{equation}
  \label{eq:stdDPG-1}
  b(\UUh,\VVh)  =  \ip{ \FF, \VVh}_h,
  \qquad
\end{equation}
for all $\VVh$ in the test space $V_h$, defined by 
\begin{equation}
    \label{eq:stdDPG-2}
  V_h = T U_h,
\end{equation}
where $T: U \to V$ is defined by 
\begin{equation}
  \label{eq:T}
  \ip{T \WW, \VV}_V =  b ( \WW,\VV),
\qquad \forall \VV \in V,
\end{equation}
and the $V$-inner product $\ip{\cdot,\cdot}_V$ is the inner product
generated by the norm
\begin{equation}
  \label{eq:Vnorm}
  \| \VV \|_V^2 = \| A_h \VV\|^2 + \veps^2\| \VV \|^2.
\end{equation}
Here $\veps > 0$ is an arbitrary scaling parameter.  Note that when
$\veps=1$,~\eqref{eq:Vnorm} defines a {\em graph norm} on $V$.  The
case $\veps=1$, analyzed in~\cite{DemkoGopalMuga11a}, is the standard
DPG method. In the next section, we will adapt the analysis
of~\cite{DemkoGopalMuga11a} to the case of the variable $\veps$, which
we refer to as the ``$\veps$-DPG method.''

It is easy to reformulate the $\veps$-DPG method as a residual
minimization problem. (All DPG methods with test spaces as
in~\eqref{eq:T} minimize a residual as already pointed out
in~\cite{DemkoGopal11}.) Letting $V'$ denote the dual space of $V$,
normed with $\| \cdot\|_{V'}$, we define $F\in V'$ by $F(\VV) =
\ip{\FF, \VV}_h$. Then letting $B: U \to V'$ denote the operator
generated by the above-defined $b(\cdot,\cdot)$, i.e., $B \WW (\VV) =
b(\WW,\VV)$ for all $\WW\in U$ and $\VV \in V$, one can immediately
see that $\UUh$ solves~\eqref{eq:stdDPG-1} if and only if
\[
\UUh = \aarg\min_{w_h \in U_h}  \| B w_h - F \|_{V'}.
\]
This norm highlights the difference between the DPG method and
the previously discussed standard $L^2$-based least-squares
method~\eqref{eq:ls}.

\subsection{Inexactly computed test spaces}   \label{ssec:inexact}

A basis for the test space $V_h$, defined in~\eqref{eq:stdDPG-2}, can
be obtained by applying $T$ to a basis of $U_h$. One application of
$T$ requires solving~\eqref{eq:T}, which although local (calculable
element by element), is still an infinite dimensional
problem. Accordingly a practical version of the $\veps$-DPG method
uses a finite dimensional subspace $V^r \subset V$ and replaces $T$ by
$T^r: U \to V^r$ defined by
\begin{equation}
  \label{eq:Tr}
  \ip{T^r \WW, \VV}_V =  b ( \WW,\VV),
\qquad \forall \VV \in V^r.
\end{equation}
In computations, we then use, in place of $V_h$, the inexactly
computed test space $V_h^r \equiv T^r U_h$, i.e., the practical DPG
method finds $\UUh^r$ in $U_h$ satisfying
\begin{equation}
  \label{eq:stdDPG-practical}
  b(\UUh^r,\VV)  =  \ip{ \FF, \VV}_h, \qquad \forall \VV \in V_h^r.
  \qquad
\end{equation}
For the Helmholtz example, we set $V^r$ as follows: Let $\QQQ_{l,m}$
denote the space of polynomials of degree at most $l$ and $m$ in $x_1$
and $x_2$, resp.  Let $RT_r \equiv \QQQ_{r,r-1} \times \QQQ_{r-1,r}$
denote the Raviart-Thomas subspace of $\Hdiv K$.  We
set 
\[
V^r = \{ v: v|_K \in RT_r \times \rev{\QQQ_{r,r} } \}.
\]
Clearly, $V^r \subseteq \Hdiv\oh \times H^1(\oh)$. Later, we shall
solve~\eqref{eq:stdDPG-practical} using $r\ge 2$ and report the
numerical results. It is easy to see using the Fortin operators
developed in~\cite{GopalQiu12a} that $T^r$ is injective for $r\ge 2$,
which implies that~\eqref{eq:stdDPG-practical} yields a positive
definite system. However, a complete analysis using~\cite{GopalQiu12a}
tracking $\og$ and $r$ dependencies, remains to be developed, and is
not the subject of this paper.

\section{Analysis of the $\veps$-DPG method}
\label{sec:analysis}

The purpose of this section is to study how the stability constant of
the $\veps$-DPG method~\eqref{eq:stdDPG-1} depends on $\veps$.  The
analysis in this section provides the theoretical motivation to
introduce the scaling by $\veps$ into the DPG setting.

\subsection{Assumption}

The analysis is under the already placed assumption that the boundary
value problem~\eqref{eq:bvp} is uniquely solvable.  We will now need a
quantitative form of this assumption. Namely, there is a constant
$C(\og)>0$, possibly depending on $\og$, such that the solution
of~\eqref{eq:bvp} satisfies
\[
    \|(\vu,\phi) \| \le C(\og) \| \FF \|.  
\]
One expects $C(\og)$ to become large as $\og$ approaches any of
the resonances in $\Sigma$.  For any $(\vr,\psi) \in R$, choosing $\FF
= A(\vr,\psi)$ and applying the above inequality, we obtain 
\begin{equation}
  \label{eq:Ck}
  \|(\vr,\psi) \| \le C(\og) \| A(\vr,\psi) \|,
  \qquad \forall (\vr,\psi) \in R.
\end{equation}
This is the form in which we will use the assumption.

Note that in the case of the impedance boundary condition, the unique
solvability assumption can be easily verified~\cite{Melen95} for all
$\og$.  Furthermore, when that boundary condition is imposed, for
instance, on the boundary of a convex domain, the
estimate~\eqref{eq:Ck} is proved in~\cite[Lemmas~4.2 and
4.3]{DemkoGopalMuga11a} using a result of~\cite{Melen95} . The
resulting constant $C(\og)$ is bounded {\em independently of
  $\og$}. However, we cannot expect this independence to hold for the
Dirichlet boundary condition~\eqref{eq:2ndorder-bc} we are presently
considering.

Finally, let us note that the ensuing analysis applies equally well to
the impedance boundary condition: We only need to replace the space
$R$ considered here by that in~\cite{DemkoGopalMuga11a} and
assume~\eqref{eq:Ck} for all functions in the revised~$R$.

\subsection{Quasioptimality}

It is well-known that if there are positive constants $C_1$ and $C_2$
such that
\begin{equation}
  \label{eq:normequiv}
  C_1 \| \VV\|_V \le 
  \sup_{ \WW\in U} \frac{ | b(\WW,\VV)| }{ \| \WW \|_U } 
  \le 
  C_2 \| \VV \|_V,\qquad \forall \VV \in V,
\end{equation}
then a quasioptimal error estimate 
\begin{equation}
  \label{eq:ee}
  \| \UU - \UUh \|_U
  \le \frac{C_2}{C_1} \inf_{\WW \in U_h} \| \UU - \WW \|_U
\end{equation}
holds.  This follows from~\cite[Theorem~2.1]{DemkoGopalMuga11a}, or
from the more general result of~\cite[Theorem~2.1]{GopalQiu12a}, after
noting that the following uniqueness condition holds: Any $\WW\in U$
satisfying $b(\WW, \VV)=0$ for all $\VV\in V$ vanishes. (Since this
uniqueness condition can be proved as
in~\cite[Lemma~4.1]{DemkoGopalMuga11a}, we shall not dwell on it
here.)

Accordingly, the remainder of this section is devoted to
proving~\eqref{eq:normequiv}, tracking the dependence of constants
with $\veps$, and using the $U$-norm we define below. First, let
\begin{align*}
  \| (\vr,\psi) \|_R  & =  \frac{1}{\veps} \| A(\vr,\psi)\|.
\end{align*}
By virtue of~\eqref{eq:Ck}, this is clearly a norm under which the
space $R$, defined in~\eqref{eq:R}, is complete.
The space $Q$ in~\eqref{eq:Q} is normed
by the quotient norm, i.e., for any $\hqq \in Q$, 
\[
\| \hqq \|_Q = \inf\left\{ \| \RR \|_R: \text{ for all }
  \RR\in R \text{ such that }  
 \tr \RR = \hqq \right\}.
\]
The function in $R$ which achieves the infimum above defines an
extension operator $E: Q \to R$ that is a continuous right inverse
of $\tr$ and satisfies 
\begin{equation}
  \label{eq:E}
\| E \hqq \|_R = \|\hqq \|_Q.
\end{equation}
With these notations, we can now define the norm on the trial space
by
\[
\| (w,\psi,\hw,\hpsi)\|_U^2 
= 
\| (w,\psi) \|^2 + \| (\hw,\hpsi) \|_Q^2.
\]
The following theorem is proved by extending the 
ideas in~\cite{DemkoGopalMuga11a} to the $\veps$-DPG method.

\begin{theorem}
  \label{thm:eps}
  Suppose~\eqref{eq:Ck} holds and let $c= C(\og)\left( C(\og)\veps/2 +
    \sqrt{ 1 + C(\og)^2 \veps^2/4} \right)$. Then the inf-sup
  condition in~\eqref{eq:normequiv} holds with
  $C_1=1/\sqrt{1+c\,\veps}$ and the continuity condition
  in~\eqref{eq:normequiv} holds with $C_2 = \sqrt{1+c\,\veps}$. Hence,
  the DPG solution admits the error estimate
  \[
  \| \UU - \UUh \|_U
  \le \left(   1+ c\,\veps \right)
  \inf_{\WW \in U_h} \| \UU - \WW \|_U.
  \]
\end{theorem}
\begin{proof}
  We first prove the continuity estimate.  Let $(\WW,\hqq) \in U$ and
  let $\VV\in V$. We use the abbreviated notations $\hqq=(\hw,\hpsi)$,
  $\WW =(w,\psi)$, and $\VV = (\vv,\eta)$.  By~\eqref{eq:Ck}
  and~\eqref{eq:E},
  \begin{equation}
    \label{eq:5}
    \| E \hqq \|\le C(\og) \veps  \| \hqq \|_Q,
    \qquad
    \| AE \hqq \|= \veps  \| \hqq \|_Q.
  \end{equation}
  The extension $E$ can be used to rewrite $b( (\WW,\hqq),
  \VV ) = -\ip{ \WW, A_h\VV }_h + \ip{
    E\hqq, A_h\VV }_h + \ip{ AE\hqq, \VV
  }_h.$ Then, applying the Cauchy-Schwarz inequality, and
  using~\eqref{eq:5}, we have
  \begin{align}
    \nonumber
    | b( (\WW,\hqq), \VV ) | 
    & \le 
    \| \WW \| \| A_h\VV \| 
    + C(\og)\veps \|\hqq \|_Q \| A_h\VV \|
    + \veps \|\hqq\|_Q  \|\VV\|
    \\ \label{eq:12}
    & \le
    \left(
      \| \WW \|^2 + \| \hqq \|_Q^2 
    \right)^{1/2} t,
  \end{align}
  where $ t^2 = \| A_h\VV\|^2 + \left( C(\og) \veps \| A_h
    \VV \| + \veps \| \VV \| \right)^2$.  With $a=C(\og)
  \veps \| A_h \VV \| $ and $b= \veps \| \VV \|$ we
  apply the inequality $(a+ b)^2 \le (1+\alpha^2) a^2 +
  (1+\alpha^{-2}) b^2$ to obtain
  \begin{align*}
    t^2
    & \le 
    \left( 1 + (1+\alpha^2)  C(\og)^2 \veps^2\right) 
    \| A_h \VV \|^2 + (1+\alpha^{-2}) \veps^2 
    \| \VV \|^2,
  \end{align*}
  for any $\alpha>0$. Setting $\alpha^2= -1/2 + \sqrt{1/4 +
    C(\og)^{-2}\veps^{-2}},$ so that 
  \begin{equation}
    \label{eq:alpha}
    (1+\alpha^2)C(\og)^2 \veps^2 = \alpha^{-2} = c\,\veps
  \end{equation}
  with $c$ as in the statement of the theorem.  Hence, $ t^2 \le
  (1+c\, \veps) \| \VV \|_V^2.  $ Returning to~\eqref{eq:12}, 
  \[
  | b( (\WW,\hqq), \VV ) | 
  \le C_2
  \| (\WW,\hqq) \|_U \| \VV \|_V.
  \]
  with $C_2 =\sqrt{1+c\,\veps}$.  This verifies the upper inequality
  of~\eqref{eq:normequiv}.

  To prove the lower inequality of~\eqref{eq:normequiv}, let $\RR$ be
  the unique function in $R$ satisfying $A\RR = \VV$ for any given
  $\VV \in V$.  Then, by~\eqref{eq:Ck},
  \begin{equation}
    \label{eq:4}
      \| \RR \| \le C(\og) \|\VV \|.
  \end{equation}
  Also, since $\| A\RR \| = \|\VV \|$, 
  letting $\hrr = \tr\RR$, we have 
  \begin{equation}
    \label{eq:7}
      \| \hrr \|_Q = 
      \frac 1 \veps \| A E\hrr \| \le \frac 1 \veps \| A\RR\|
      = \frac 1 \veps \|\VV\|.
  \end{equation}
  By~\eqref{eq:2},  we have $ \ip{ A\RR, \VV }_h = -\ip{\RR, A_h\VV }_h +
  \iip{\hrr, \VV }_h$, so 
  \begin{equation}
    \label{eq:11}
    \| \VV\|_V^2 
    = 
    \veps^2 \| \VV \|^2 + \| A_h\VV \|^2
    =    \veps^2\,  b( (\ZZ,\hrr), \VV),
  \end{equation}
  where $\ZZ = \RR - \veps^{-2} A_h\VV$, a function that can be bounded
  using~\eqref{eq:4}, as follows:
  \begin{align*}
    \| \ZZ \|^2 
    & \le 
    (1+\alpha^2) \| \RR \|^2 
    + (1 + \alpha^{-2} ) \veps^{-4} \| A_h\VV\|^2
    \\
    & \le 
     (1+\alpha^2)C(\og)^2 \| \VV \|^2
    + (1 + \alpha^{-2} ) \veps^{-4} \| A_h\VV\|^2,
  \end{align*}
  for any $\alpha>0$. Choosing $\alpha$ as in~\eqref{eq:alpha} and
  using~\eqref{eq:4}--\eqref{eq:7},
  \begin{align}
    \nonumber
    \veps^4\| (\ZZ,\hrr)  \|_U^2 
    & = \veps^4\| \ZZ\|^2 +\veps^4 \| \hrr\|_Q^2 
    \\ \nonumber
    & 
    \le  
    \left( 1+ (1+\alpha^2)C(\og)^2\veps^2\right)\veps^2 \| \VV \|^2
    + (1 + \alpha^{-2} ) \| A_h\VV\|^2
    \\ \label{eq:13}
    & 
    \le (1+c\, \veps)
    \left( \veps^2 \|\VV \|^2 + \| A_h\VV \|^2\right).
  \end{align}
  Returning to~\eqref{eq:11}, we now have
  \begin{align*}
     \| \VV\|_V^2 
     &
     = 
     \frac{b( (\ZZ,\hrr), \VV) }
     {\|  (\ZZ,\hrr) \|_U}  \,\veps^2\,  \| (\ZZ,\hrr) \|_U
     \le \left( \sup_{\XX\in U} \frac{ |b( \XX,\VV)| }{\|\XX\|_U}
     \right) \sqrt{ 1+c\, \veps}\, \| \VV\|_V
  \end{align*}
  by virtue of~\eqref{eq:13}, verifying the lower
  inequality of~\eqref{eq:normequiv} with $C_1 =1/\sqrt{ 1+c\,
    \veps}.$ 
\end{proof}

\begin{remark}
  Although we presented the above result only for the Helmholtz
  equation, the ideas apply more generally.  It seems possible to
  prove a similar result abstractly, e.g., using the abstract setting
  in~\cite{Bui-TDemkoGhatt11}, for any DPG application that uses a
  scaled graph norm analogous to~\eqref{eq:Vnorm} (with the wave
  operator $A_h$ replaced by suitable others).
\end{remark}

\subsection{Discussion}

Theorem~\ref{thm:eps} shows that the use of the $\veps$-scaling in the
test norm can ameliorate some stability problems, e.g., those that can
arise from large $C(\omega)$.

Observe that the best possible value for the constant $C_2/C_1$
in~\eqref{eq:ee} is $1$. Indeed, if $C_2/C_1$ equals $1$, then the
computed solution~$\UUh$ coincides with the best approximation to
$\UU$ from $U_h$. Theorem~\ref{thm:eps} shows that the quasioptimality
constant of the DPG method approaches the ideal value of $1$ as
$\veps \to 0$. 

However, since the norms depend on $\veps$, we must
further examine the components of the error separately, by defining
\begin{subequations}
  \label{eq:e-he}
\begin{align}
  \eint^2
  & = \| \vu - \vuh \|^2  + \| \phi - \phi_h \|^2,
  \\ \label{eq:6}
  \he^2
  & = \| A E ( \hu - \hu_h,\hphi - \hphi_h) \|^2.
\end{align}
\end{subequations}
The estimate of Theorem~\ref{thm:eps} implies that 
\begin{equation}
  \label{eq:9}
e^2 + \frac{\he^2}{\veps^2} 
\le
\left(1 + c \,\veps \right)^2
\left(a^2 + \frac{\ha^2}{\veps^2} 
\right)
\end{equation}
where $a$ and $\ha$ are the best approximation errors defined by 
\begin{align}
\label{eq:18}
  a^2 & = \inf_{(\vw,\psi,0,0)\in U_h} \| \vu - \vw \|^2 + \| \phi - \psi \|^2, \\
\nonumber
  \ha^2 & = \inf_{(0,0,\hw,\hpsi)\in U_h} \| AE(\hu - \hw, \hphi - \hpsi)\|^2.
\end{align}
Note that $E$ is independent of $\veps$.

We want to compare the error bounds for the numerical fluxes and
traces in the $\veps=1$ case with the case of $0<\veps \ll 1$. To
distinguish these cases we will denote the error defined
in~\eqref{eq:6} by $\he_1$ when $\veps=1$.  Clearly,~\eqref{eq:9}
implies
\begin{align}
  \label{eq:3}
  \he_1^2
  & \le 
  \left(1+c\right)^2
  \left(
    a^2 + \ha^2
  \right).
\end{align}
For the other case,~\eqref{eq:9} implies, after multiplying through by
$\veps^2$,
\[
\he^2 \le  \left(1+c\,\veps\right)^2\left( \veps^2a^2 + \ha^2 \right)\rev{.}
\]
Comparing this with~\eqref{eq:3}, and noting that $a$ and $\ha$ remain
the same for different $\veps$, we find that the DPG errors for fluxes
and traces admit a {\em better bound for smaller $\veps$.}  Whether
the actually observed numerical error improves, will be investigated
through the dispersion analysis presented in a later section, as well
as in the next subsection.

\subsection{Numerical illustration}

Theorem~\ref{thm:eps} partially explains a numerical observation we
now report. We implemented the $\veps$-DPG method by setting the
parameter $r=3$ (see \S~\ref{ssec:inexact}) and computed
$\UUh^r=(\vuh^r,\phi_h^r,\hu_h^r,\hphi_h^r)$.  In analogy
with~\eqref{eq:e-he}, define the discretization errors $\eint_r$ and
$\he_r$ by $ \eint^2_r= \| \vu - \vuh^r \|^2 + \| \phi - \phi_h^r \|^2
$ and $ \he^2_r= \| A E ( \hu - \hu_h^r,\hphi - \hphi_h^r) \|^2.  $
Although Theorem \ref{thm:eps} suggests an investigation of
\[
 \frac{\| \UU - \UUh^r \|_U}{\inf_{\WW \in U_h} \| \UU - \WW \|_U}
=\left ( \frac{e^2_r + (\he_r/\veps)^2 }{a^2 + (\ha/\veps)^2}\right )^{1/2}, 
\]
due to the difficulty of applying the extension operator $E$ in
practice, we have investigated the ratio $\eint_r/a$ as a function of
$\og$. Recall that $a$ is the $L^2(\om)$ best approximation error
defined in~\eqref{eq:18}, so $\eint_r/a$ measures how close the
discretization errors are to the best possible.

\begin{figure}
  \begin{center}
    \includegraphics[width=0.5\textwidth]{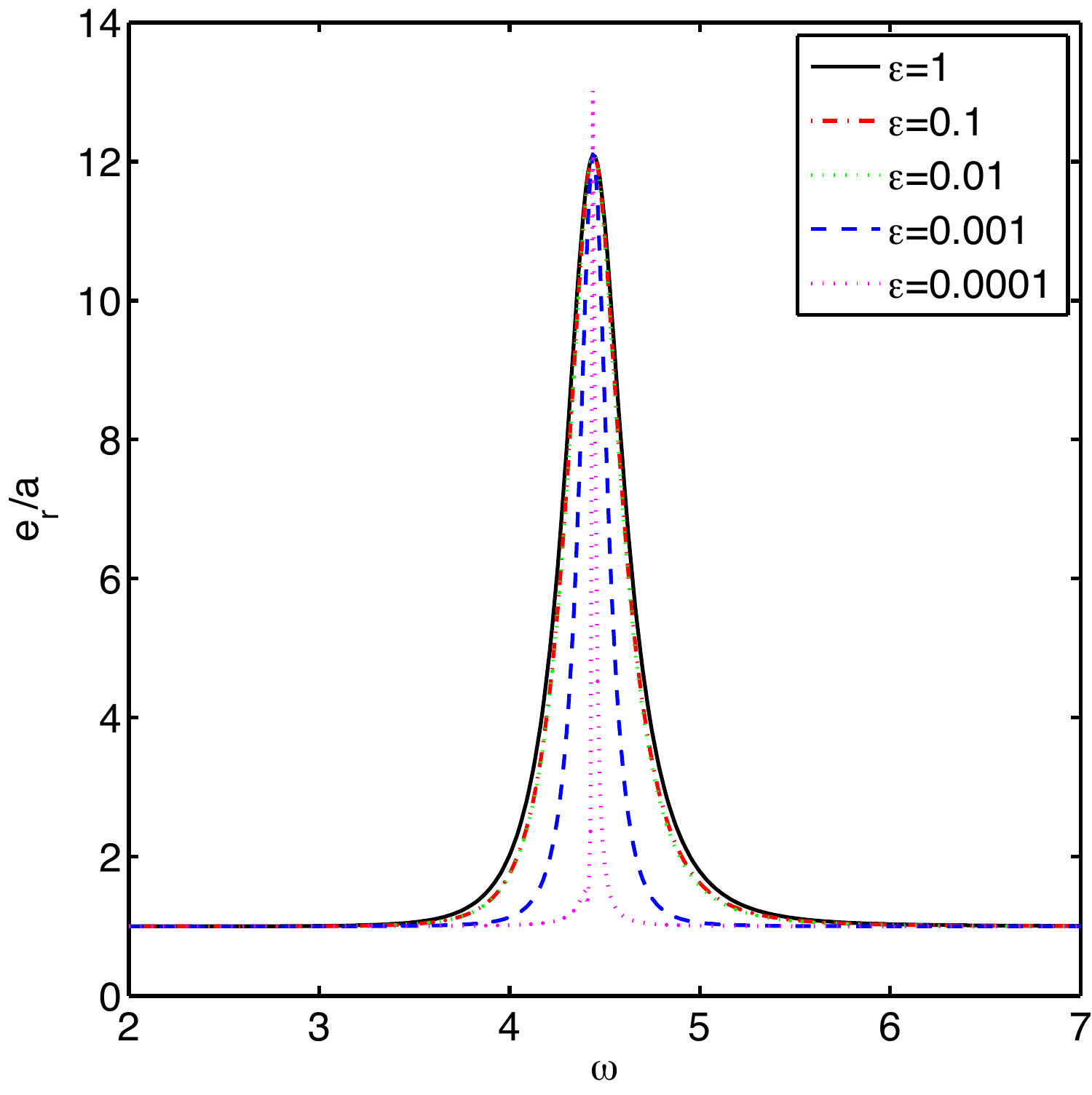}
  \end{center}
  \caption{The regularizing effect of $\veps$-DPG method as seen  from 
    a plot of  the ratio $e_r/a$ near a resonance.}
  \label{fig:ratio}
\end{figure}

For a range of wavenumbers $\omega$, we chose the data $\FF=(\vec 0,
f)$ so that the exact solution to \eqref{eq:bvp} on the unit square
would be $(\vu,\phi)=(-\frac{i}{\omega}\grad\phi,\phi)$, with
$\phi=x(1-x)y(1-y)$. Each resulting boundary value problem was then
solved using the $\veps$-DPG method with $\veps=10^{-n}, n=0,1,2,3,4$,
on a fixed mesh of $h=1/16$ and the corresponding discretization
errors $\eint_r$ were collected.

The resulting ratios $\eint_r/a$ are plotted as a function of $\og$ in
Figure~\ref{fig:ratio} for a few $\veps$ values.  First of all,
observe that the graph of the ratio begins close to the optimal value
of one for all $\veps$ values in the figure.  Next, observe that the
ratio spikes up as $\og$ approaches the exact resonance value $\og =
\pi \sqrt 2 \approx 4.44$, where $C(\og)$ is infinity.  It is
interesting to look at the points near (but not at) the
resonance. Observe that as $\veps$ is decreased, the DPG method
exhibits a ``regularizing'' effect at points near the resonance: E.g.,
at $\og=5$, the values of $\eint_r/a$ are closer to $1$ for smaller
$\veps$. It therefore seems advantageous to use smaller $\veps$ for
problems near resonance.

The theoretical explanation for this numerical observation would be
complete (by virtue of Theorem~\ref{thm:eps}), if we had computed
using the exact DPG test spaces ($r=\infty$), instead of the inexactly
computed spaces ($r=3$).  Certain discrete effects arising due to this
inexact computation of test spaces will be presented in a later
section. 

\section{Lowest order stencil}  
\label{sec:stencil}

We now consider the example of square two-dimensional elements.  The
lowest order case of the DPG method is obtained using $Q(\d K) = \{
(\hw,\hpsi) :$ $\hw$ is constant on each edge of $\d K$, $\hpsi$ is
linear on each edge of $\d K$, and $\hpsi$ is continuous on $\d K\}$.
Let $S(K) =\{ (\vw, \psi):\; \vw$ and $\psi$ are constants (vector and
scalar, resp.) functions on $K\}$.  We consider the DPG method (with
$\veps$) using the lowest order global trial space 
\[
U_h = S_h \times Q_h,
\]
where $ Q_h = \{ \hrr \in Q: \hrr|_{\d K} \in Q(\d K)$ for all mesh
elements $K \}$ and $S_h = \{ \WW : \WW|_K \in S(K) $ for all mesh
elements $K\}$.

Let $\hat\chi_e$ denote the indicator function of an edge $e$. If $a$
denotes a vertex of the square element $K$, let $\phi_a$ denote the
bilinear function that equals one at $a$ and equals zero at the other
three vertices of $K$. Let $\hphi_a = \phi_a|_{\d K}$. The collection
of eight functions of the form $(0,\hphi_a)$ and $(\hat\chi_e,0)$, one
for each vertex, and one for each edge of $K$, forms a basis for~$Q(\d
K)$.  We distinguish between the horizontal and vertical edges,
because the unknowns there approximate different components of the
velocity~$\vu$. Accordingly, we will denote by $\hat\chi_e^h$ the
indicator function of a horizontal edge and by $\hat\chi_e^v$ the
indicator function of a vertical edge.

The local $11 \times 11$ element DPG matrix is defined using a basis
for $Q(\d K)$ and $S(K)$. (While a basis for $Q(\d K)$ is obtained as
mentioned above, a basis for $S(K)$ is trivially obtained by three
indicator functions.) If we enumerate the 11 basis functions as $e_i$,
$i=1,\ldots 11$, then the element DPG matrix is defined by 
\begin{equation}
  \label{eq:B}
  B_{ij} = b( e_j, T^r e_i )
\end{equation}
where $T^r$ is as defined in~\eqref{eq:Tr}. Since this matrix depends
on $\og$ and $\veps$, we will write $B \equiv B(\og, \veps)$.  In our
computations, we do not use any specialized basis for $V^r$ to compute
the action of $T^r$, so to overcome round-off problems due to
ill-conditioned local matrices, we resorted to high precision
arithmetic for these local computations.

To show how $B$ can be computed by mapping, let $\breve K=
[0,1]^2$. For any square $K$ of side length $h$, there is a
translation vector $\vb_K$ such that the $K -\vb_K = h \breve K.$ For
any (scalar or vector) function $\VV$ on $K$, let $\breve \VV$ on
$\breve K$ be the mapped function obtained by $\breve \VV (\breve x) =
\VV( h\breve x + \vb_K)$. Let us denote the matrix computed
using~\eqref{eq:B}, but using the mapped basis functions $\breve e_i$
on $\breve K$, by $\breve B(\og,\veps)$. Then by a change of
variables, it is easy to see that
\begin{equation}
  \label{eq:scaleB}
  B(\og,\veps) = h^2 \breve B (\og h, \veps h).
\end{equation}
Thus we may compute local DPG matrix by scaling the DPG matrix on the
fixed reference element $\breve K$ obtained using the {\em normalized
  wavenumber} $\og h$ and scaling parameter $\veps h$.  It is enough
to compute the element matrix $\breve B$ using high precision
arithmetic for the ensuing dispersion analysis.

Next, we eliminate the three interior variables of $S(K)$ and consider
the {\em condensed} $8\times 8$ element stiffness matrix for the
variables in $Q(\d K)$. At this stage it will be useful to classify
these eight variables (unknowns) into three categories:
\begin{inparaenum}
\item Unknowns at vertices $a$ 
  (which are the coefficients multiplying the basis function 
  $\hphi_a$) 
  denoted by ``\bullt'', 
\item unknowns on  horizontal edges (coefficients multiplying
  $\hat\chi_e^h$) denoted by ``\;\uparr\;'', and 
\item unknowns on vertical edges (coefficients
  multiplying the corresponding $\hat\chi_e^v$) denoted by ``\rtarr''.
\end{inparaenum}
The normal vectors on all horizontal and vertical
edges are fixed to be $(0,1)$ and $(1,0)$, respectively,
corresponding to the direction of the above-indicated arrows.

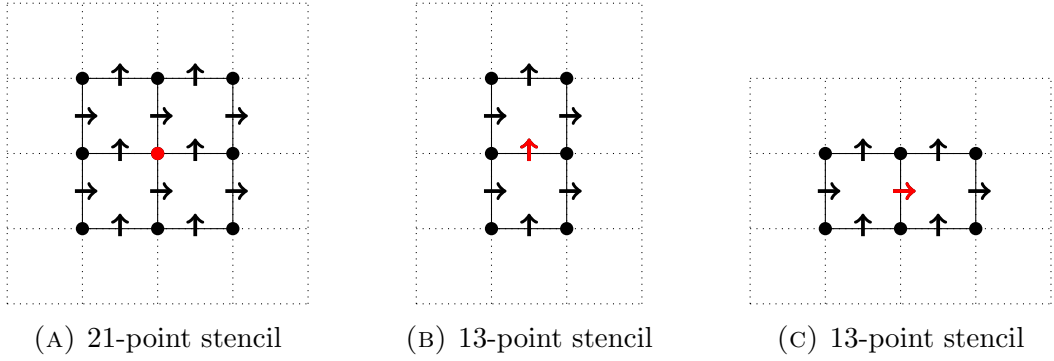
\begin{figure}
  \centering
  \begin{subfigure}[b]{0.3\textwidth}
    \centering
    \begin{tikzpicture}
      \draw [step = 1cm,dotted] (0,0) grid (4,4);
      \draw [step = 1cm] (1,1) grid (3,3);
      \foreach \x in {1,2,3}  {
        \fill (\x,1) circle (2.5pt);
        \fill (\x,2) circle (2.5pt);
        \fill (\x,3) circle (2.5pt);      
      }
      \foreach \x in {1,2,3}  {
        \draw[->,line width=0.05cm] (\x-0.1,1.5)--(\x+0.2,1.5);
        \draw[->,line width=0.05cm] (\x-0.1,2.5)--(\x+0.2,2.5);
      }
      \foreach \y in {1,2,3}  {
        \draw[->,line width=0.05cm] (1.5,\y-0.1)--(1.5,\y+0.2);
        \draw[->,line width=0.05cm] (2.5,\y-0.1)--(2.5,\y+0.2);
      }
      \fill[red] (2,2) circle (2.5pt);
    \end{tikzpicture}
    \caption{21-point stencil}
    \label{fig:stencils21}
  \end{subfigure}
  \begin{subfigure}[b]{0.3\textwidth}
    \centering
    \begin{tikzpicture}
      \draw [step = 1cm,dotted] (0,0) grid (3,4);
      \draw [step = 1cm] (1,1) grid (2,3);
      \foreach \x in {1,2}  {
        \fill (\x,1) circle (2.5pt);
        \fill (\x,2) circle (2.5pt);
        \fill (\x,3) circle (2.5pt);      
      }
      \foreach \x in {1,2}  {
        \draw[->,line width=0.05cm] (\x-0.1,1.5)--(\x+0.2,1.5);
        \draw[->,line width=0.05cm] (\x-0.1,2.5)--(\x+0.2,2.5);
      }
      \foreach \y in {1,2,3}  {
        \draw[->,line width=0.05cm] (1.5,\y-0.1)--(1.5,\y+0.2);
      }
      \draw[->,red,line width=0.05cm] (1.5,2-0.1)--(1.5,2+0.2);
    \end{tikzpicture}
    \caption{13-point stencil}
    \label{fig:stencils13v}
  \end{subfigure}
  \begin{subfigure}[b]{0.3\textwidth}
    \centering
    \begin{tikzpicture}
      \draw [step = 1cm,dotted] (0,0) grid (4,3);
      \draw [step = 1cm] (1,1) grid (3,2);
      \foreach \x in {1,2,3}  {
        \fill (\x,1) circle (2.5pt);
        \fill (\x,2) circle (2.5pt);
      }
      \foreach \x in {1,2,3}  {
        \draw[->,line width=0.05cm] (\x-0.1,1.5)--(\x+0.2,1.5);
      }
      \draw[->,red,line width=0.05cm] (2-0.1,1.5)--(2+0.2,1.5);
      \foreach \y in {1,2}  {
        \draw[->,line width=0.05cm] (1.5,\y-0.1)--(1.5,\y+0.2);
        \draw[->,line width=0.05cm] (2.5,\y-0.1)--(2.5,\y+0.2);
      }
    \end{tikzpicture}
    \caption{13-point stencil}
    \label{fig:stencils13h}
  \end{subfigure}
  \caption{Stencils}
  \label{fig:stencils}
\end{figure}

Now suppose the mesh is a uniform mesh of congruent square
elements. Assembling the above-described condensed $8\times 8$ element
matrices on such a mesh, we obtain a global system where the interior
variables are all condensed out.  The resulting equations can be
represented using the stencils in Figure~\ref{fig:stencils}. A row of
the matrix system corresponding to an unknown of the type ``\bullt'',
connects to unknowns of the same type at other vertices, as well as
unknowns of the other two types, as shown in the 21-point stencil in
Figure~\ref{fig:stencils21}.  Similarly, the unknowns of the type
``\;\uparr\;'' and ``\rtarr'' connect to other unknowns in the
13-point stencils depicted in Figures~\ref{fig:stencils13v}
and~\ref{fig:stencils13h}, respectively. These stencils will form the
basis of our dispersion analysis next.

\section{Dispersion analysis}

This section is devoted to a numerical study of the DPG method with
$\veps$, by means of a dispersion analysis.  The dispersion analysis
is motivated by~\cite{DeraeBabusBouil99}. Details on dispersion
analyses applied to standard methods can be found in~\cite{Ainsw04}
and the extensive bibliography presented therein.

\subsection{The approach}

To briefly adapt the approach of~\cite{DeraeBabusBouil99} to fit our
context, we consider a general method for the homogeneous Helmholtz
equation on an infinite uniform lattice $(h\ZZZ)^2$.  Suppose the method
has $S$ different types of nodes on this lattice, some falling in
between the lattice points, each corresponding to a different type of
variable, with its own stencil (and hence its own equation).  All
nodes of the $t^\mathrm{th}$ type $(t=1,2,\ldots, S)$ are assumed to
be of the form $\vj h$ where $\vj$ lies in an infinite subset of
$(\ZZZ/2)^2$.  The solution value at a general node $\vj h$ of the
$t^\mathrm{th}$ type is denoted by $\pjt$. Note that methods with
multiple solution components are accommodated using the above
mentioned node types.

The $t^\mathrm{th}$ stencil, centered around $\vj h$, consists of a
finite number of nodes, some of which belong to the $t^\mathrm{th}$
stencil, and the remaining belong to other stencils. Suppose we have
finite index sets $J_s \subset (\ZZZ/2)^2$, for each $s =1,2,\ldots,
S$, such that all the nodes of the $t^\mathrm{th}$ stencil centered
around $\vj h$ can be listed as $ N_{\vj,t} = \{ (\vj + \vl) h: \; \vl
\in J_s \text{ and } s=1,2,\ldots, S\}$ with the understanding that
$(\vj+\vl)h$ is a node of $s^\mathrm{th}$ type whenever $\vl \in J_s$.
This allows interaction between variables of multiple types.  Every
node $(\vj+\vl)h$ in $N_{\vj,t}$ has a corresponding stencil
coefficient (or weight) denoted by $\Dtsl$. Due to translational
invariance, these weights do not change if we place the stencil at
another center node $\vj' h$, hence the numbers $\Dtsl$ do not depend
on the center index~$\vj$.

We obtain the method's equation at a general node $\vj h$ of the
$t^\mathrm{th}$ type by applying the $t^\mathrm{th}$ stencil centered
around $\vj h$ to the solution values $\{\pjt\}$, namely
\begin{equation}
  \label{eq:stencilr}
  \sum_{s=1}^S \sum_{l  \in J_s}
  \Dtsl\,  \pjsl =0.
\end{equation}
Note that we have set all sources to zero to get a zero right hand
side in~\eqref{eq:stencilr}.

Plane waves, $\psi(\vx) \equiv e^{\ii \vk \cdot \vx}$, are exact
solutions of the Helmholtz equation with zero sources (and are often
used to represent other solutions). Here the wave vector $\vk$ is of
the form $ \vk = \omega (\cos(\theta), \sin(\theta) ) $ for some $0\le
\theta <2 \pi$ representing the direction of propagation.  The
objective of dispersion analysis is to find similar solutions of the
discrete homogeneous system.  Accordingly, we set
in~\eqref{eq:stencilr}, the ansatz
\begin{equation}
  \label{eq:8}
  \pjt = a_t e^{\ii \vk_h \cdot \vj h},
\end{equation}
where $\vk_h = \omega_h (\cos(\theta),\sin(\theta))$ and $a_t$ is an
arbitrary complex number associated to the $t^\mathrm{th}$ variable
type. We want to find such discrete wavenumbers $\og_h$
satisfying~\eqref{eq:stencilr}.

To this end, we must solve~\eqref{eq:stencilr} after
substituting~\eqref{eq:8} therein, namely
\begin{equation}
  \label{eq:14}
  \sum_{s=1}^S a_s\sum_{l  \in J_s}
  \Dtsl    e^{\ii \vk_h \cdot (\vj+\vl)h} =0, 
\end{equation}
for all $t = 1,2,\ldots S$.  Multiplying by $e^{-\ii \vk_h \cdot
  \vj h}$, we remove any dependence on $\vj$.  Defining
the $S\times S$ matrix $F \equiv (F_{ts}(\og_h))$ by
\[
F_{ts}(\og_h) = 
\sum_{\vl  \in J_s}
  \Dtsl   e^{\ii (\omega_h (\cos\theta,\sin\theta) \cdot \vl) h},
\]
we observe that solving~\eqref{eq:14} is equivalent to solving
\begin{equation}
  \label{eq:detF}
  \det F(\og_h) = 0.
\end{equation}
This is the nonlinear equation we solve to obtain the discrete
wavenumber~$\og_h$ corresponding to any given $\theta$ and $\omega$.

\subsection{Application to the DPG method}

Next, we apply the above-described framework to the lowest order DPG stencil
discussed in Section~\ref{sec:stencil}. Since there are three
different types of stencils (see Figure~\ref{fig:stencils}), we have
$S=3$. The first type of unknowns, denoted by ``\bullt'', represent
the DPG method's approximation to the value of $\phi$ at the nodes
$\vj h$ for all $\vj \in \ZZZ^2$. The stencil of the first type is the
one shown in Figure~\ref{fig:stencils21}. The unknowns of the second
type represent the method's approximation to the vertical components
of $\vu$ on the midpoints of horizontal edges, i.e., at all points in
$(h \ZZZ + h/2 ) \times h\ZZZ$. These unknowns were previously denoted
by ``\;\uparr\;'' and has the stencil portrayed in
Figure~\ref{fig:stencils13v}. Similarly, the third type of stencil and
unknown are as in Figure~\ref{fig:stencils13h}. To
summarize,~\eqref{eq:8} in the lowest order DPG case, becomes
\begin{align*}
  \psi_{1,\vj} &  = \hphi_h(\vx_{\vj}) = a_1 e^{\ii \vk_h\cdot \vx_{\vj}}
  && \forall \vx_{\vj}  \in (h\ZZZ)^2,
  \\
  \psi_{2,\vj} & = \hu_h(\vx_{\vj}) = a_2 e^{\ii \vk_h\cdot \vx_{\vj}}
  && \forall \vx_{\vj}\in (h \ZZZ + h/2 ) \times h\ZZZ,
  \\
  \psi_{3,\vj} & = \hu_h(\vx_{\vj}) = a_3 e^{\ii \vk_h\cdot \vx_{\vj}}
  && \forall \vx_{\vj}\in  h\ZZZ \times (h \ZZZ + h/2 ).
\end{align*}
The condensed $8
\times 8$ DPG matrices, discussed in Section~\ref{sec:stencil}, can be
used to compute the stencil weights $D_{t,s,\vl}$ in each of the
three cases, which in turn lead to the $3\times 3$ nonlinear
system~\eqref{eq:detF} for any given propagation angle $\theta$.

We numerically solved the nonlinear system for $\og_h$, for various
choices of $\theta$ (propagation angle), $r$ (enrichment degree), $
\veps$ (scaling factor in the $V$-norm), and $h$ (mesh size).  The
first important observation from our computations is that the computed
wavenumbers $\og_h$ are complex numbers. They lie close to $\og$ in
the complex plane. The small but nonzero imaginary parts of $\og_h$
indicate that the DPG method has dissipation errors, in addition to
dispersion errors.  The results are described in more detail below.

\subsection{Dependence on $\theta$}

To understand how dispersion errors vary with propagation angle
$\theta$, we fix the exact wave number~$\og$ appearing in the
Helmholtz equation to $1$ (so the wavelength is $2\pi$) and examine
the computed $\re(\og_h)$ for each $\theta$.

\begin{figure}   
  \centering     
  \begin{subfigure}{0.49\textwidth}
    \begin{center}
      \includegraphics[width=0.95\textwidth]{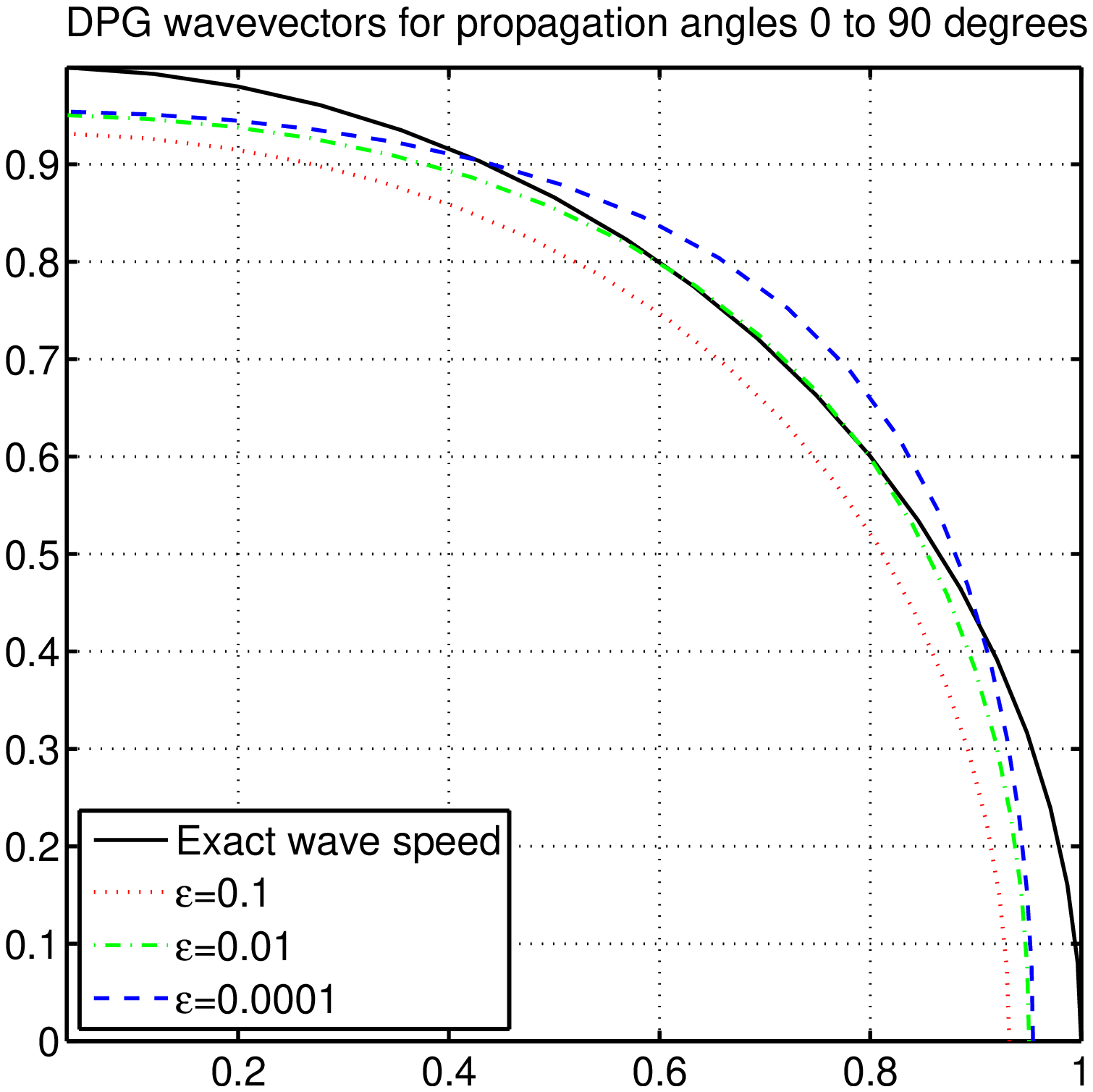}
      \caption{Dependence of $\re(\vk_h)$ on $\veps$}
      \label{fig:wavevec:epsilons}
    \end{center}
  \end{subfigure}
  \begin{subfigure}{0.49\textwidth}
    \begin{center}
      \includegraphics[width=0.95\textwidth]{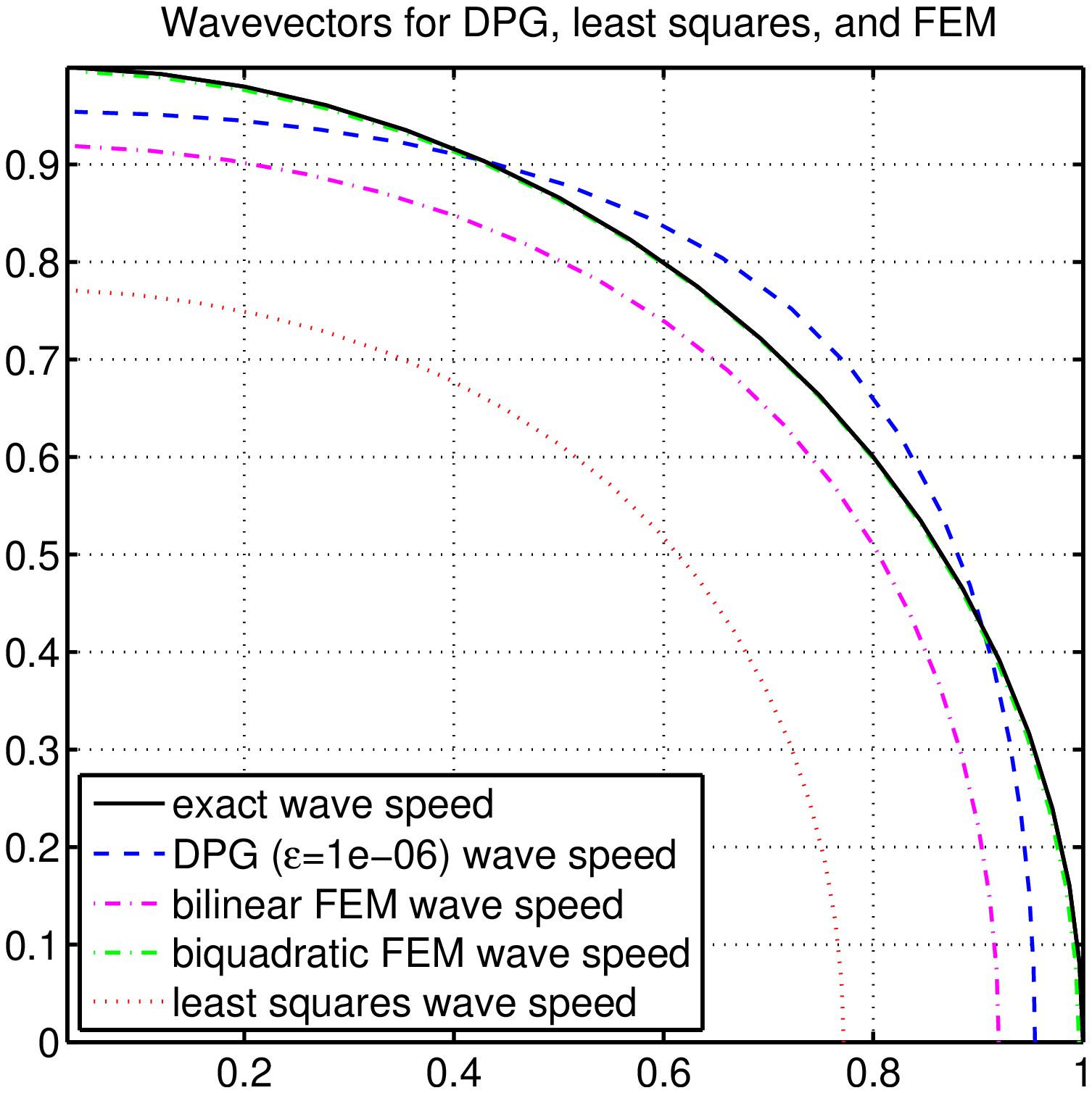}
      \caption{Comparison of methods}
      \label{fig:wavevec:comp}
    \end{center}
  \end{subfigure}
  \caption{The curves traced out by the discrete wavevectors $\vk_h$
    as $\theta$ goes from $0$ to $\pi/2$. These plots were obtained
    using $\omega=1$ and $h=2\pi/4$.}
  \label{fig:wavevec}
\end{figure}

One way to visualize the results is through a plot of the
corresponding discrete wavevectors $\re(\vk_h)$ vs.~$\vk$ for every
propagation direction~$\theta$. Due to symmetry, we only need to
examine this plot in the region $0\le \theta \le \pi/2$.  We present
these plots for the case $r=3$ in Figure~\ref{fig:wavevec}.  We fix
$h=2\pi/4$. (This corresponds to four elements per wavelength if the
propagation direction is aligned with a coordinate axis.)  In
Figure~\ref{fig:wavevec:epsilons}, we plot the curve traced out by the
endpoints of the discrete wavevectors $\vk_h$.  We see that as $\veps$
decreases, the curve gets closer to the (solid) circle traced out by
the exact wavevector~$\vk$.  This indicates better control of
dispersive errors with decreasing $\veps$ (cf.~Theorem~\ref{thm:eps}).

In Figure~\ref{fig:wavevec:comp}, we compare the $\vk_h$ obtained
using the lowest order DPG method with the discrete wavenumbers of the
standard lowest order (bilinear) finite element method (FEM).  Clearly
the wavespeeds obtained from the DPG method are closer to the exact
$\og=1$ than those obtained by bilinear FEM. However, since the lowest
order DPG method has a larger stencil than bilinear FEM, one may argue
that a better comparison is with methods having the same stencil
size. We therefore compare the DPG method with two other methods which
have exactly the same number of points in their stencil: (i)~The
biquadratic FEM, which after condensation has three stencils of the
same size as the lowest order DPG method, and (ii)~the conforming
first order $L^2(\om)$ least-squares method using the lowest order
Raviart-Thomas and Lagrange spaces (which has no interior nodes to
condense out). While the wavespeeds from the DPG method did not
compare favorably with the biquadratic FEM of~(i), we found that the
DPG method performs better than the least-squares method in~(ii).

\subsection{Dependence on $\veps$ and $r$}

\begin{figure}     
  \begin{center}   
    \begin{subfigure}{\textwidth}
    \begin{center}
      \includegraphics[width=\textwidth,clip=true,trim=4cm 0 4cm 0]
      {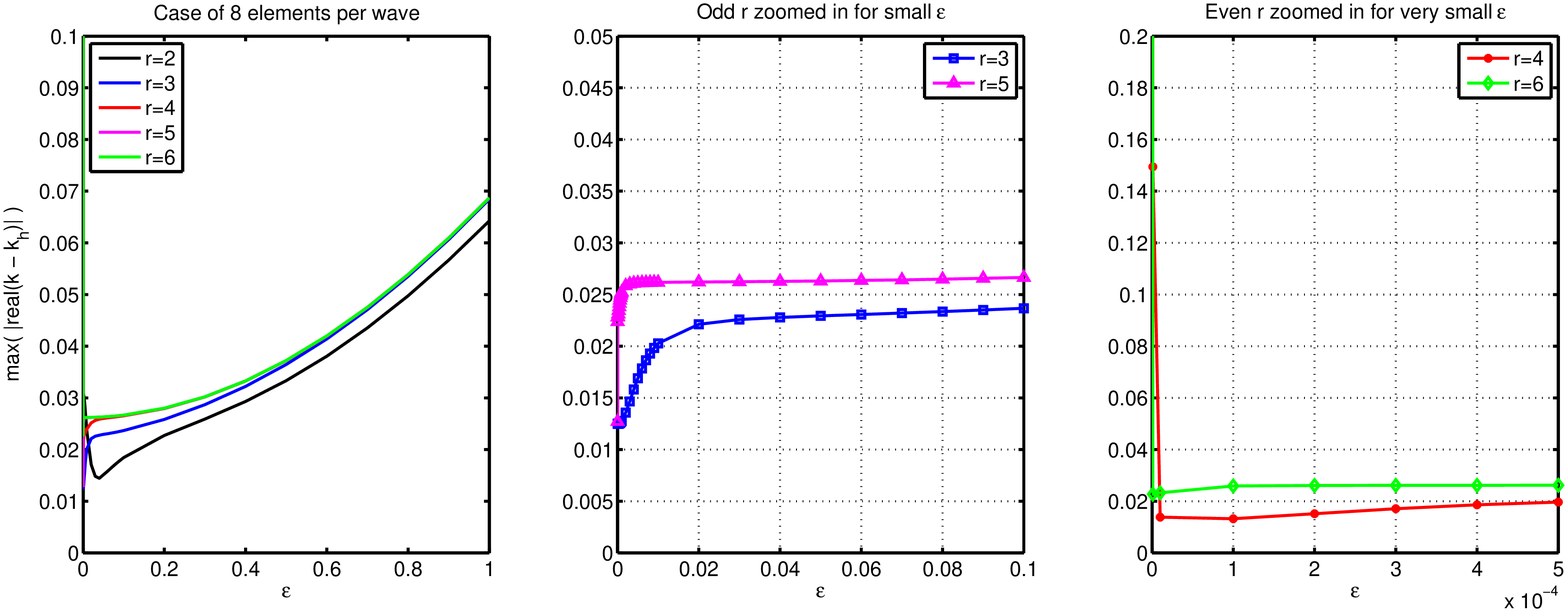}
    \end{center}
    \caption{Dispersive errors: Plots of $\rho$ vs.~$\veps$}
    \label{fig:disperive8}
  \end{subfigure}
  \begin{subfigure}{\textwidth}
    \begin{center}
      \includegraphics[width=\textwidth,clip=true,trim=4cm 0 4cm 0]
      {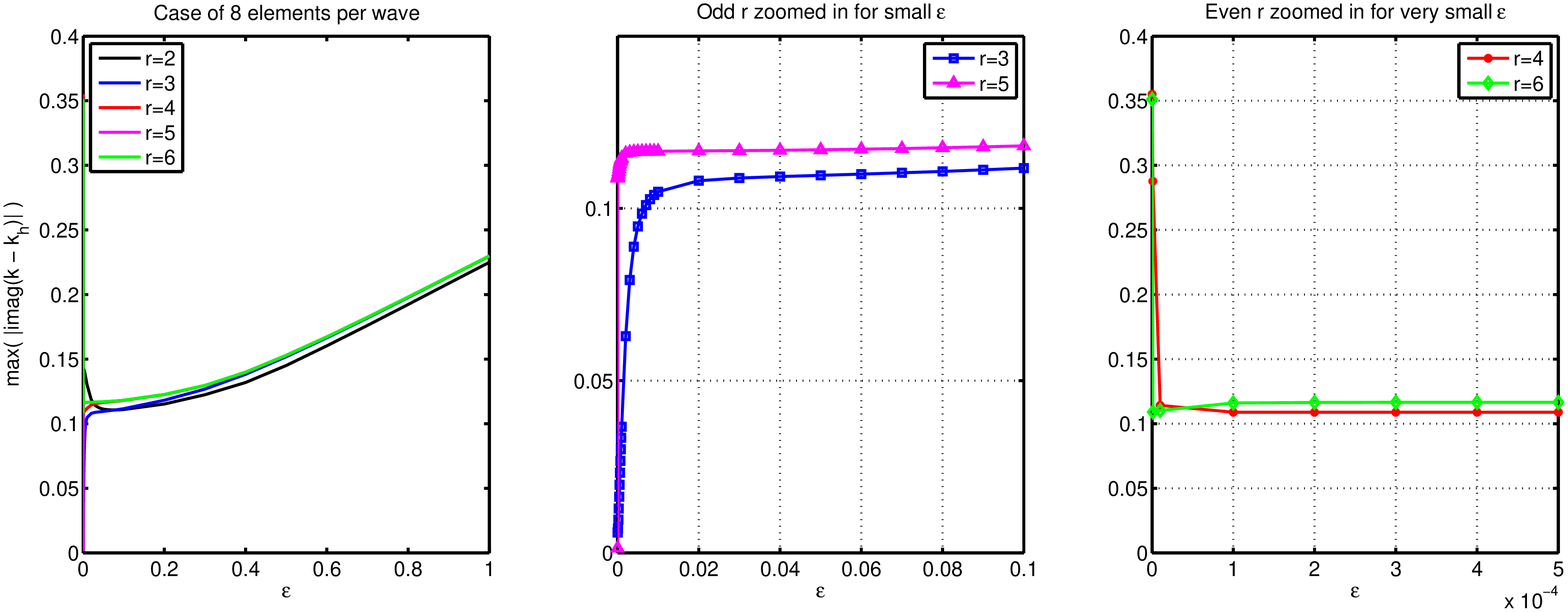}
    \end{center}
    \caption{Dissipative errors: Plots of $\eta$ vs.~$\veps$}
    \label{fig:dissipative8}
  \end{subfigure}
\end{center}
\caption{The discrepancies between exact and discrete wavenumbers as a 
  function of $\veps$, when $\og=1$ and $h=2\pi/8$.}
\label{fig:8elts}
\end{figure}

We have seen in Figure~\ref{fig:wavevec} that the discrete wavespeed
$\og_h$ is a function of the propagation angle $\theta$.  We now
examine the maximum discrepancy between real and imaginary parts of
$\og_h$ and $\og$ over all angles. Accordingly, define
\[
\rho = \max_\theta \left| \re( \og_h(\theta)) - \og \right|,
\qquad
\eta = \max_\theta \left| \im( \og_h(\theta))  \right|.
\]
The former indicates dispersive errors while the latter indicates
dissipative errors. Fixing $\og=1$ and $h =2\pi/8$ (so that there are
about eight elements per wavelength), we examine these quantities as a
function of $r$ and $\veps$ in Figure~\ref{fig:8elts}. The first of
the plots in Figures~\ref{fig:disperive8} and~\ref{fig:dissipative8}
show that the errors decrease as $\veps$ decreases from 1 to about
0.1. In view of Theorem~\ref{thm:eps}, we expected this decrease. 

However, the behavior of the method for smaller $\veps$ is curious. In
the remaining plots of Figure~\ref{fig:8elts} we see that when $r$ is
odd, the errors continue to decrease for smaller $\veps$, while for
even $r$, the errors start to increase as $\veps \to 0$. This suggests
the presence of discrete effects due to the inexact computation of
test functions. We do not yet understand it enough to give a
theoretical explanation.

\subsection{Dependence on $\omega$}

\begin{figure}     
  \begin{subfigure}{0.48\textwidth}
    \begin{center}
      \includegraphics[width=\textwidth]
      {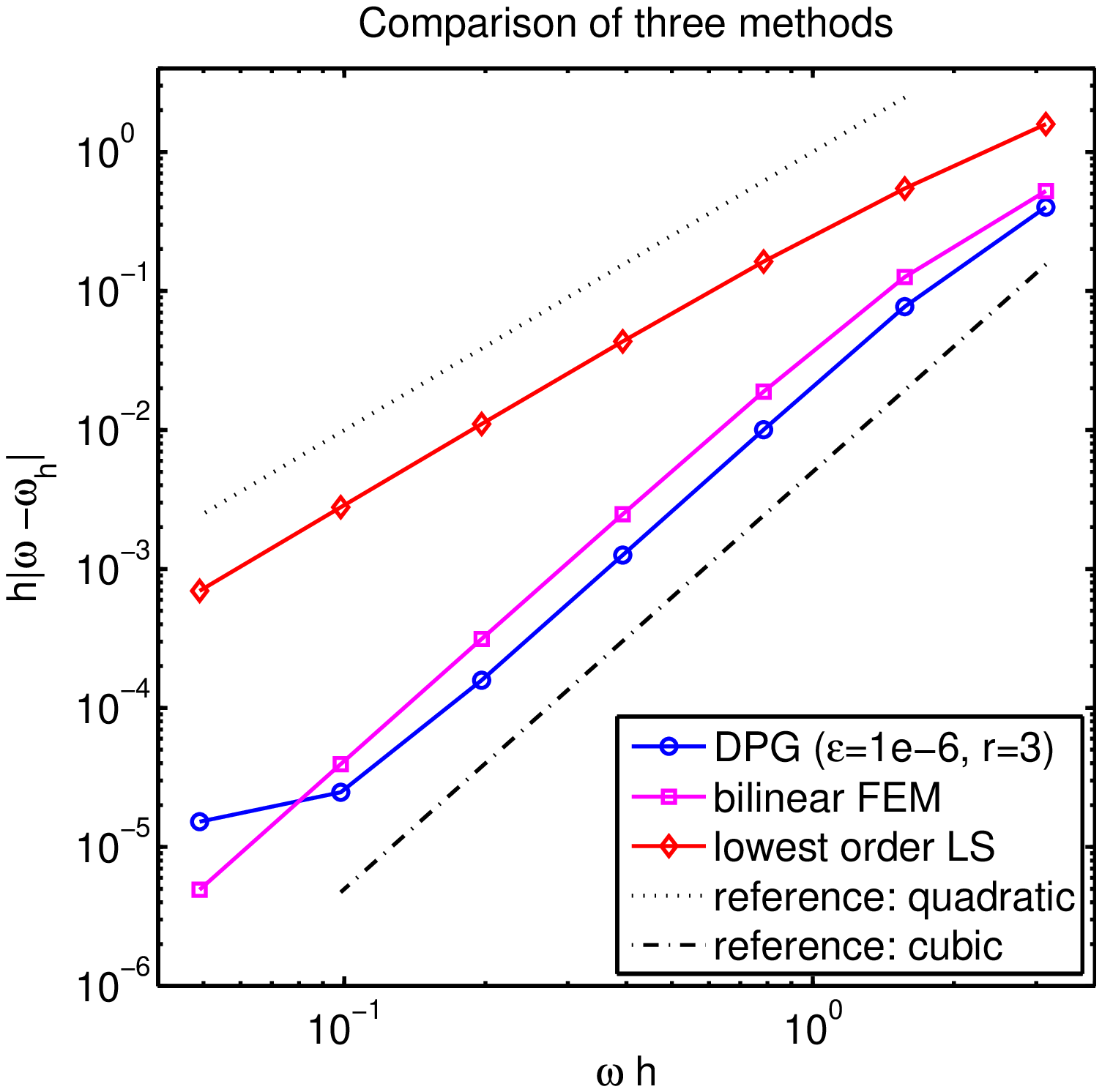}
    \end{center}
    \caption{Plot of $|\og_h h-\og h|$ for three methods}
    \label{fig:cgcecomp}
  \end{subfigure}\quad
  \begin{subfigure}{0.48\textwidth}
    \begin{center}
      \includegraphics[width=\textwidth]
      {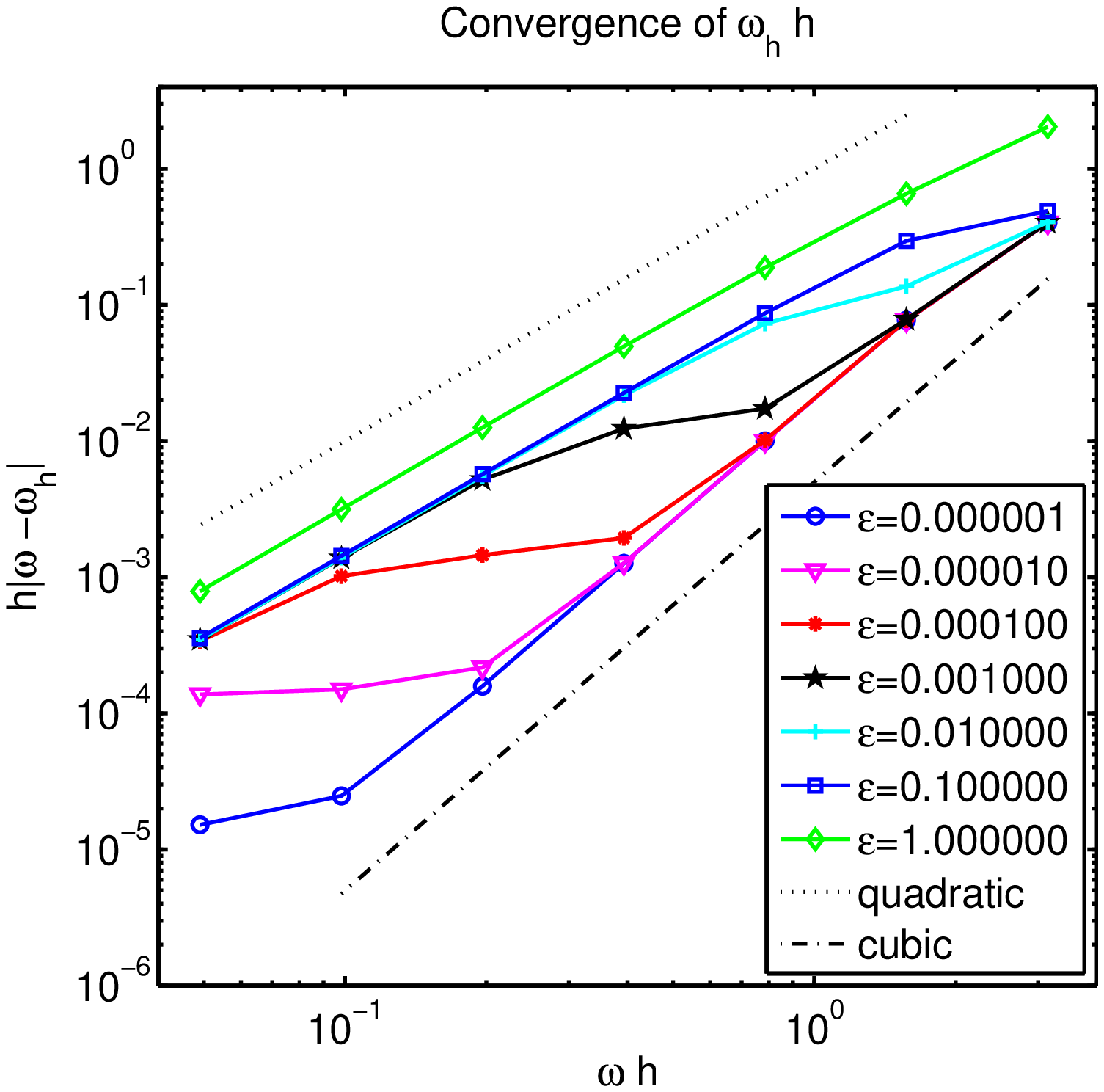}
    \end{center}
    \caption{Case of DPG with $r=3$ and various $\veps$}
    \label{fig:cgceepsl}
  \end{subfigure}
  \caption{Rates of convergence of $|\og_hh-\og h|$ to zero for small $\og h$,
  in the case of propagation angle $\theta=0$.}
\label{fig:cgce}
\end{figure}

Now we examine how $\omega_h$ depends on $\omega$. First, let us note
that the matrix $F$ in~\eqref{eq:detF} only depends on $\og h$. (This
can be seen, for instance, from~\eqref{eq:scaleB} and noting how the
stencil weights depend on the entries of $B$.) Hence, we will study
how $\omega_h h$ depends on the normalized wavenumber $\og h$,
restricting ourselves to the case of $\theta=0$.

In Figure~\ref{fig:cgcecomp}, we plot (in logarithmic scale) the
absolute value of $\og_h h - \og h$ vs.~$\og h$ for the standard
bilinear FEM, the lowest order $L^2$ least-squares method (marked LS),
and the DPG method with $\veps=10^{-6}, r=3$. We observe that while
$|\og_h h - \og h |$ appears to decrease at $O(\og h)^2$ for the least
squares method, it appears to decrease at the higher rate of $O(\og
h)^3$ for the FEM and DPG cases considered in the same graph.  For
easy reference, we have also plotted lines indicating slopes
corresponding to $O(\og h)^2$ and $O(\og h)^3$ decrease, marked
``quadratic'' and ``cubic'', resp., in the same graph. 

Note that a convergence rate of $|\og_h h - \og h | = O(\og h)^3$
implies that the difference between discrete and exact wave speeds
goes to zero at the rate
\[
| \og_ h - \og | = \og \;O ( \og h)^2.
\]
This shows the presence of the so-called~\cite{BabusSaute00} pollution
errors: For instance, as $\og$ increases, even if we use finer meshes
so as to maintain $\og h$ fixed, the error in wave speeds will
continue to grow at the rate of $O(\og)$. Our results show that
pollution errors are present in all the three methods considered in
Figure~\ref{fig:cgcecomp}. The difference in convergence rates, e.g.,
whether $| \og_ h - \og |$ converges to zero at the rate $\og \;O (
\og h)^2$ or at the rate $\og \;O ( \og h),$ becomes important, for
example, when trying to answer the following question: What $h$ should
we use to obtain a fixed error bound for $| \og_h - \og |$ for all
frequencies $\omega$?  While methods with convergence rate $\og O(\og
h)$ would require $h \approx \og^{-2}$, methods with convergence rate
$\og O(\og h)^2$ would only require $h \approx \og^{-3/2}$.

Next, consider Figure~\ref{fig:cgceepsl}, where we observe interesting
differences in convergence rates within the DPG family.  While the DPG
method for $\veps=1$ exhibits the same quadratic rate of convergence
as the least-squares method, we observe that a transition to higher
rates of convergence progressively occur as $\veps$ is decreased by
each order of magnitude. The $\veps=10^{-6}$ case shows a rate
virtually indistinguishable from the cubic rate in the considered
range.  The convergence behavior of the DPG method thus seems to vary ``in
between'' those of the least-squares method and the standard FEM as
$\veps$ is decreased. The values of $\og h$ considered in these plots
are $2\pi/2^l$ for $l=1,2,\ldots, 7$, which cover the numbers of
elements per wavelength in usual practice.

\begin{figure}     
  \begin{center} 
    \begin{subfigure}{0.48\textwidth}
    \begin{center}
      \includegraphics[width=\textwidth]
      {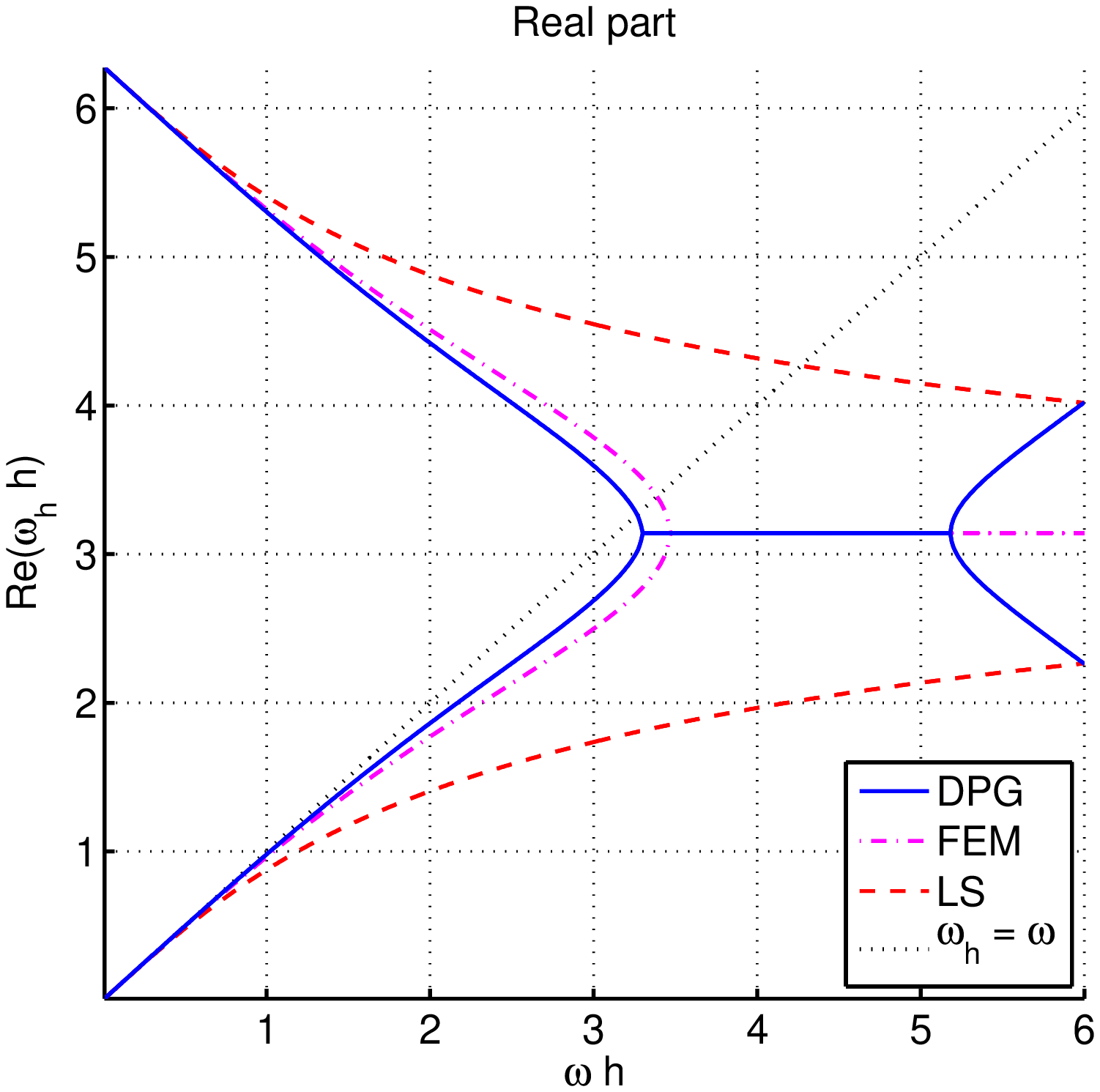}
    \end{center}
    \caption{$\re(\og_h h)$ as a function of $\omega h$}
    \label{fig:realband}
  \end{subfigure}\quad
  \begin{subfigure}{0.48\textwidth}
    \begin{center}
      \includegraphics[width=\textwidth]
      {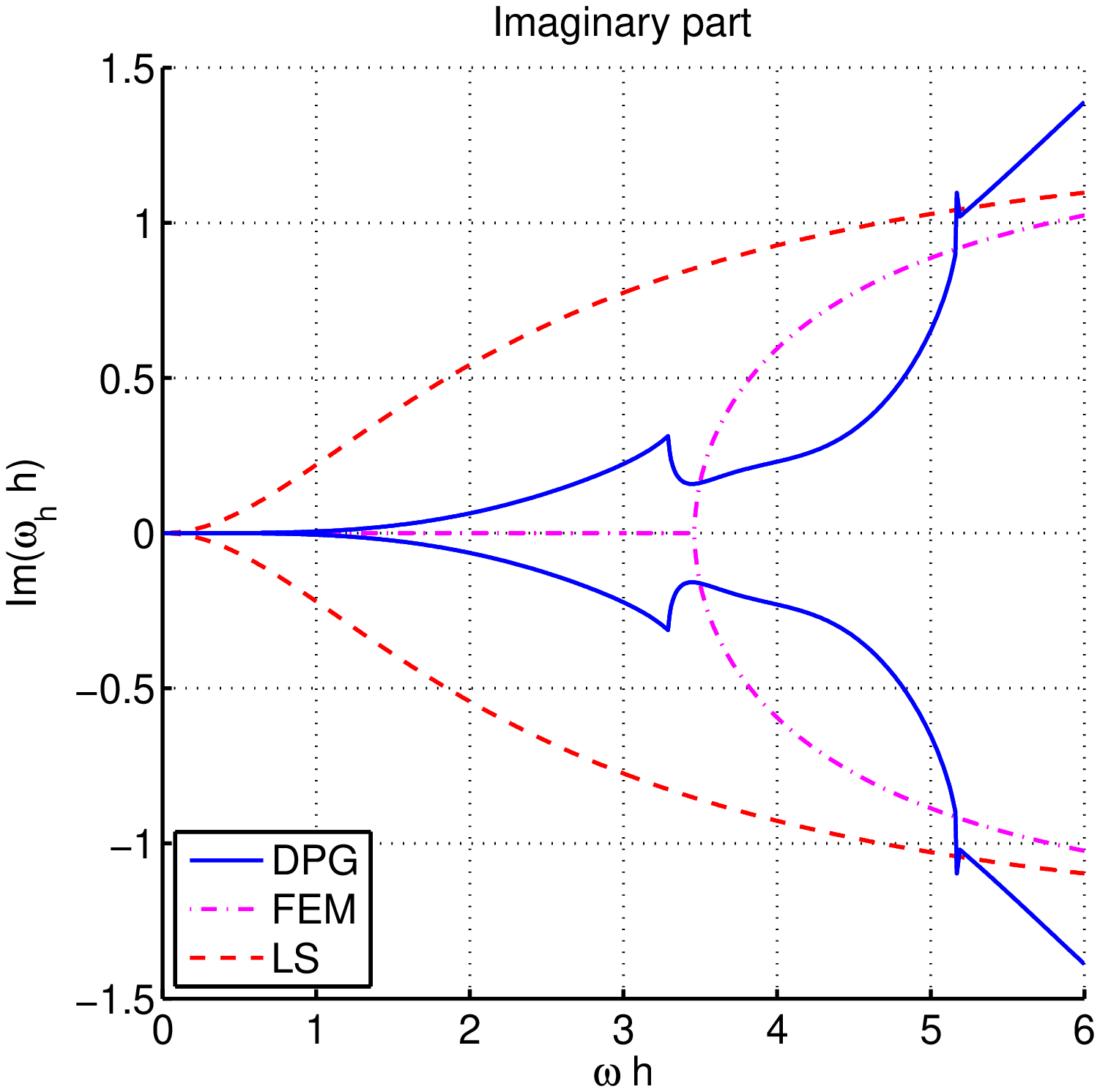}
    \end{center}
    \caption{$\im(\og_h h)$ as a function of $\omega h$}
    \label{fig:imagband}
  \end{subfigure}
\end{center}
\caption{A comparison of discrete wavenumbers obtained by three lowest
  order methods in the case of propagation angle $\theta=0$.}
\label{fig:bands}
\end{figure}

Next, we consider a wider range of $\og h$
following~\cite{ThompPinsk94}, where such a study was done for
standard finite elements, separating the real and imaginary parts of
$\og_h h$. Our results for the case of $\theta=0$ are collected in
Figure~\ref{fig:bands}. To discuss these results, let us first recall
the behavior of the standard bilinear finite element method (whose
discrete wavenumbers are also plotted in dash-dotted curve in
Figure~\ref{fig:bands}).  From its well-known dispersion relation (see
e.g,~\cite{Ainsw04}), we observe that if $\og_h h$ solves the
dispersion relation, then $2\pi -\og_h h$ also solves it. Accordingly, 
the plot in Figure~\ref{fig:realband} is symmetric about the
horizontal line at height $\pi$. Furthermore, as already shown
in~\cite{ThompPinsk94}, $\og_h h$ is real-valued in the range $0<
\omega h < \sqrt{12}$. The threshold value $\og h = \sqrt{12}$ was
called the ``cut-off'' frequency.  (Note that in the regime $\og h >
\pi$, we have less than two elements per wavelength. Note also that
$\sqrt{12} > \pi$.)  As can be seen from Figures~\ref{fig:realband}
and~\ref{fig:imagband}, in the range $\sqrt{12}< \og h\le 6$, the
bilinear finite elements yield $\og_h h$ with a constant real part of
$\pi$ and nonzero imaginary parts of increasing magnitude.

We observed a somewhat similar behavior for the DPG method -- see the
solid curves of Figure~\ref{fig:bands}, which were obtained after
calculating $F$ explicitly using the computer algebra package Maple,
for the lowest order DPG method, setting $r=3$ and $\veps=0$. The major
difference between the DPG and FEM results is that $\og_h h$ from the
DPG method was not real-valued even in the regime where FEM
wavenumbers were real. It seems difficult to define any useful
analogue of the cut-off frequency in this situation. Nonetheless, we
observe from Figures~\ref{fig:realband} and~\ref{fig:imagband} that
there is a segment of constant real part of value $\pi$, before which
the imaginary part of $\og_h h$ is relatively small. As the number of
mesh elements per wavelength increases (i.e., as $\og h$ becomes
smaller), the imaginary part of $\og_h h$ becomes small. We therefore
expect the diffusive errors in the DPG method to be small when $\og h$
is small. Finally, we also conclude from Figure~\ref{fig:bands} that
both the dispersive and dissipative errors are better behaved for the
DPG method when compared to the $L^2$ least-squares method.

\section{Conclusions}

We presented and analyzed the $\veps$-DPG method for the Helmholtz
equation.  The case $\veps=1$ was analyzed previously
in~\cite{DemkoGopalMuga11a}.  The numerical results
in~\cite{DemkoGopalMuga11a} showed that in a comparison of the ratio
of $L^2$ norms of the discretization error to the best approximation
error is compared, the DPG method had superior properties. The
pollution errors reported in~\cite{DemkoGopalMuga11a} for a higher
order DPG method were so small that its growth could not be
determined conclusively there.  In this paper, by performing a
dispersion analysis on the DPG method for the lowest possible order,
we found that the method has pollution errors that asymptotically grow with
$\og$ at the same rate as other comparable methods.

In addition, we found both dispersive and dissipative type of errors
in the lowest order DPG method. The dissipative errors manifest in
computed solutions as artificial damping of wave amplitudes (e.g., as
illustrated in Figure~\ref{fig:dissip}).

Our results show that the DPG solutions have higher accuracy than an
$L^2$-based least-squares method with a stencil of identical
size. However, the errors in the (lowest order) DPG method did not
compare favorably with a standard (higher order) finite element method
having a stencil of the same size. Whether this disadvantage can be
offset by the other advantages of the DPG methods (such as the
regularizing effect of $\veps$, and the fact that it yields Hermitian
positive definite linear systems and good gradient approximations)
remains to be investigated.

We provided the first theoretical justification for considering the
$\veps$-modified DPG method. If the test space were exactly computed,
then Theorem~\ref{thm:eps} shows that the errors in numerical fluxes
and traces will improve as $\veps\to 0$. However, if the test space is
inexactly computed using the enrichment degree $r$, then the numerical
results from the dispersion analysis showed that errors continually
decreased as $\veps$ was decreased only for odd  $r$. A full
theoretical explanation of such discrete effects and the limiting
behavior when $\veps$ is $0$ deserves further study.







\end{document}